\documentclass[letterpaper,11pt,oneside,reqno]{amsart}

\usepackage[pdftex,backref=page,colorlinks=true,linkcolor=blue,citecolor=red]{hyperref}
\usepackage[alphabetic,nobysame]{amsrefs}

\usepackage{amsmath,amssymb,amsthm,amsfonts,mathtools}
\usepackage{graphicx,color}
\usepackage{upgreek}
\usepackage{esint}
\usepackage[mathscr]{euscript}
\usepackage[shortlabels]{enumitem}
\setlist[enumerate]{leftmargin=*}

\allowdisplaybreaks
\numberwithin{equation}{section}

\usepackage{tikz}
\usepackage{pgfplots}
\pgfplotsset{compat=1.18}
\usetikzlibrary{shapes,arrows,positioning,decorations.markings}

\usepackage{ytableau}
\usepackage{array}
\usepackage{adjustbox}
\usepackage{aliascnt}
\usepackage{cleveref}
\usepackage[margin=10pt]{caption}
\usepackage{booktabs}

\usepackage[DIV=12]{typearea}

\newcommand{\ssp}{\hspace{1pt}}
\renewcommand{\le}{\leqslant}
\renewcommand{\ge}{\geqslant}

\newcommand{\Prob}{\operatorname{\mathbb{P}}}

\DeclareMathOperator{\Des}{Des}
\DeclareMathOperator{\maj}{maj}

\newcommand{\brelbow}[2]{
  \draw[line width=2.5pt] ({#1+0.5},{#2}) to[out=90,in=180] ({#1+1},{#2+0.5});}
\newcommand{\bjelbow}[2]{
  \draw[line width=2.5pt] ({#1},{#2+0.5}) to[out=0,in=270] ({#1+0.5},{#2+1});}
\newcommand{\bvert}[2]{
  \draw[line width=2.5pt] ({#1+0.5},{#2}) -- ({#1+0.5},{#2+1});}
\newcommand{\bhoriz}[2]{
  \draw[line width=2.5pt] ({#1},{#2+0.5}) -- ({#1+1},{#2+0.5});}
\newcommand{\bcross}[2]{
  \draw[line width=2.5pt] ({#1},{#2+0.5}) -- ({#1+1},{#2+0.5});
  \draw[line width=2.5pt] ({#1+0.5},{#2}) -- ({#1+0.5},{#2+1});}
\newcommand{\bbox}[2]{
  }

\newcommand{\excise}[1]{}

\newtheorem{proposition}{Proposition}[section]
\newaliascnt{lemma}{proposition}
\newtheorem{lemma}[lemma]{Lemma}
\aliascntresetthe{lemma}
\newaliascnt{corollary}{proposition}
\newtheorem{corollary}[corollary]{Corollary}
\aliascntresetthe{corollary}
\newaliascnt{theorem}{proposition}
\newtheorem{theorem}[theorem]{Theorem}
\aliascntresetthe{theorem}
\newaliascnt{conjecture}{proposition}
\newtheorem{conjecture}[conjecture]{Conjecture}
\aliascntresetthe{conjecture}
\newaliascnt{claim}{proposition}

\aliascntresetthe{claim}
\newaliascnt{question}{proposition}
\newtheorem{question}[question]{Question}
\aliascntresetthe{question}

\crefname{proposition}{proposition}{propositions}
\Crefname{proposition}{Proposition}{Propositions}
\crefname{lemma}{lemma}{lemmas}
\Crefname{lemma}{Lemma}{Lemmas}
\crefname{corollary}{corollary}{corollaries}
\Crefname{corollary}{Corollary}{Corollaries}
\crefname{theorem}{theorem}{theorems}
\Crefname{theorem}{Theorem}{Theorems}
\crefname{conjecture}{conjecture}{conjectures}
\Crefname{conjecture}{Conjecture}{Conjectures}
\crefname{claim}{claim}{claims}
\Crefname{claim}{Claim}{Claims}
\crefname{question}{question}{questions}
\Crefname{question}{Question}{Questions}

\theoremstyle{definition}
\newaliascnt{definition}{proposition}
\newtheorem{definition}[definition]{Definition}
\aliascntresetthe{definition}
\newaliascnt{remark}{proposition}
\newtheorem{remark}[remark]{Remark}
\aliascntresetthe{remark}
\newaliascnt{example}{proposition}

\aliascntresetthe{example}
\newaliascnt{observation}{proposition}

\aliascntresetthe{observation}
\crefname{definition}{definition}{definitions}
\Crefname{definition}{Definition}{Definitions}
\crefname{remark}{remark}{remarks}
\Crefname{remark}{Remark}{Remarks}
\crefname{example}{example}{examples}
\Crefname{example}{Example}{Examples}
\crefname{observation}{observation}{observations}
\Crefname{observation}{Observation}{Observations}

\begin{document}

\title{Computation and sampling for Schubert specializations}

\author{David Anderson} \address{Ohio State University, Columbus, OH} \email{anderson.2804@math.osu.edu} \author{Greta Panova}\address{University of Southern California, Los Angeles, CA}\email{gpanova@usc.edu}  \author{Leonid Petrov} \address{University of Virginia, Charlottesville, VA}\email{lenia.petrov@gmail.com}
\date{March 20, 2026}
\thanks{DA was partially supported by NSF grant DMS-1945212 and by a Membership at the Institute for Advanced Study funded by the Charles Simonyi Endowment. GP was partially supported by NSF grant CCF:AF-2302174. LP was partially supported by NSF grant DMS-2153869 and by the Simons Foundation Travel Support for Mathematicians Awards. Part of this research was performed while GP and LP were visiting the Institute for Pure and Applied Mathematics (IPAM), supported by NSF Grant No.\ DMS-1925919. The authors acknowledge Research Computing at the University of Virginia for providing computational resources and technical support (\url{https://rc.virginia.edu}).}

\begin{abstract}
We present computational results related to
principal specializations of
the Schubert polynomials $\mathfrak{S}_w(1^n)$ for permutations $w\in S_n$.
Equivalently, these specializations count reduced pipe dreams
(and reduced bumpless pipe dreams -- RBPD) with boundary conditions determined by $w$.
We find the first counterexample, at $n=17$, to the conjecture of Merzon-Smirnov~\cite{merzon2016determinantal} that the maximal value of $\mathfrak{S}_w(1^n)$ is obtained at a layered permutation. However, the simulations suggest that $\lim_{n \to \infty} \log(\max_{w\in S_n}\mathfrak{S}_w(1^n))/n^2$ is the same constant arising for layered permutation from~\cite{MoralesPakPanova2019}.  Simultaneously, we explore the typical permutation obtained from uniformly random RBPDs, i.e. drawn from the distribution proportional to $\mathfrak{S}_w(1^n)$. Our simulations reveal a permuton-like asymptotic behavior similar to the one derived for the analogous problem for Grothendieck polynomials in~\cite{GrothendieckShenanigans2024}. 

We implement and compare the performance of three recurrence relations for computing
$\mathfrak{S}_w(1^n)$:
the descent formula of Macdonald, the transition formula of
Lascoux--Sch\"utzenberger, and the cotransition formula of Knutson.
We investigate Markov chain algorithms for sampling uniformly random
reduced bumpless pipe dreams (whose number is $\sum_{w\in S_n} \mathfrak{S}_w(1^n)$).
We prove a negative result:
the global constraint of reducedness
breaks the sublattice property of the underlying
alternating sign matrix (ASM) lattice,
which prevents the use of the standard monotone
Coupling From The Past (CFTP),
and leads to false coalescence of the extremal chains.
To bypass this,
we develop a highly efficient Markov chain
Monte Carlo (MCMC)
sampler,
augmented with macroscopic ``droop'' updates
to guarantee state space connectivity and accelerate mixing.
Our implementations enable
computation of $\mathfrak{S}_w(1^n)$ for permutations $w$ in $S_n$ up to
$n\sim 20$ on a personal computer
(and beyond on a computing cluster),
as well as uniform sampling
of reduced bumpless pipe dreams
up to $n\sim 60$ on a personal computer
(and $n\sim 100$ on a cluster).
\end{abstract}

\maketitle

\setcounter{tocdepth}{1}
\tableofcontents
\setcounter{tocdepth}{4}

\section{Introduction}
\label{sec:intro}

It is difficult to find a black cat in a dark room, especially if the cat is
not there. Likewise, it is difficult to prove a conjecture, especially if it is
wrong. Here we investigate the saga of the maximal principal specializations of
Schubert polynomials, initiated by Stanley in the aptly named ``shenanigans''
paper~\cite{stanley2017some} which asked for the leading term asymptotics of
the maximal principal specialization of a Schubert polynomial. Asymptotically,
this is the same as the number of reduced bumpless pipe dreams (RBPDs), a
six-vertex model with long-range interaction. 
Merzon and Smirnov~\cite{merzon2016determinantal} conjectured that the maximum of principal specializations
is achieved by a layered permutation, 
and Morales, Pak, and Panova~\cite{MoralesPakPanova2019} determined the optimal layered permutation and its asymptotics.
For
$n\leq 13$, the maximizer is indeed a layered permutation.

In this paper we discover
a slightly larger Schubert polynomial for $n=17$, thereby disproving the
conjectured maximal family. At the same time, we explore the typical
permutations sampled from uniformly random RBPDs, that is, permutations $w$
drawn with probabilities proportional to the principal specialization
$\mathfrak{S}_w(1^n)$ of their Schubert polynomials. We experimentally derive
the conjectured permuton and limit shapes for the RBPD. The computational data
resembles the behavior shown to hold for the analogous Grothendieck polynomials
in~\cite{GrothendieckShenanigans2024} and lends support to the conjectured asymptotic
behavior. Obtaining these experimental results is a significant computational
challenge which we tackle in this paper. It is still not practically possible
to search for the permutation maximizing the Schubert polynomial for values of
$n>20$, but the counterexamples and evidence we produce show that such
permutations may not have a simple defining structure, yet the asymptotic growth of the principal Schubert specializations is the same as for layered. 

\subsection{Background and definitions}
Let $S_n$ denote the group of permutations of $n$ elements.
Schubert polynomials $\mathfrak{S}_w$,
indexed by $w\in S_n$ and depending on $n$ variables
$x_1,\ldots,x_n$,
are fundamental objects in algebraic combinatorics and algebraic geometry which give a basis for the homology of flag varieties.  For a thorough survey of the role of Schubert polynomials in enumerative geometry and algebraic combinatorics, see \cite{BilleyGaoPawlowski2025}.

Schubert polynomials admit two combinatorial interpretations of significance to statistical mechanics, namely through tilings with certain global constraints.

The first, introduced by Bergeron and Billey
\cite{BilleyBergeron} and Fomin and Kirillov \cite{FK},
represents Schubert polynomials as generating functions over
\emph{reduced pipe dreams} (also called RC-graphs).
These are configurations of crossing and elbow tiles in a
staircase shape, satisfying the constraint that each pair of pipes
crosses at most once (the \emph{reduced} condition).
To each reduced pipe dream one can associate its \emph{boundary permutation} $w\in S_n$
by tracing where the pipes of all colors exit at the top.
See \Cref{fig:pipe_dreams_intro} for an illustration.

\begin{figure}[htbp]
\centering
\includegraphics[width=0.5\textwidth]{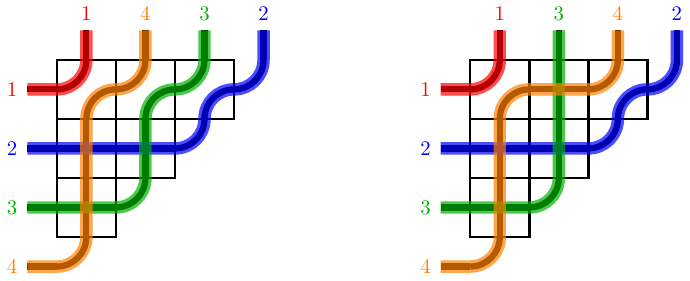}
\caption{Pipe dreams for $n=4$.
{Left:} A reduced pipe dream for $w = 1432$ where each pair of pipes crosses at most once.
{Right:} A non-reduced (forbidden) pipe dream where pipes 3 and 4 cross twice.
Color is added merely as a visual aid to highlight the pipes' paths.}
\label{fig:pipe_dreams_intro}
\end{figure}

More recently,
Schubert polynomials were interpreted as generating functions over
objects of another type,
\emph{reduced bumpless pipe dreams} (RBPD, for short).
This interpretation was promoted by
Lam, Lee, and Shimozono \cite{LamLeeShimozono2021BPD},
and further developed by Weigandt \cite{Weigandt2020_bumpless}, who also observed a connection with earlier work by Lascoux \cite{Lascoux02ice}.  
Bumpless pipe dreams are tilings of an $n\times n$ square grid with six types of tiles,
subject to the same global \emph{reduced} condition that each pair of pipes
crosses at most once. 
Without the reducedness constraint,
bumpless pipe dreams are the same as configurations of the
six-vertex (square ice) model with domain wall boundary conditions.
These configurations are well-known to be in bijection with alternating sign matrices,
see, e.g., Bressoud \cite{Bressoud1999} and
Zinn-Justin \cite{ZinnJustin20096Vertex}.

Similarly to the ordinary pipe dreams in the staircase,
to each RBPD one can associate its \emph{boundary permutation} $w\in S_n$
by tracing where the pipes of all colors exit through the right side.
See \Cref{fig:bpd_examples_intro} for an illustration.

\begin{figure}[htbp]
\centering
\includegraphics[width=0.6\textwidth]{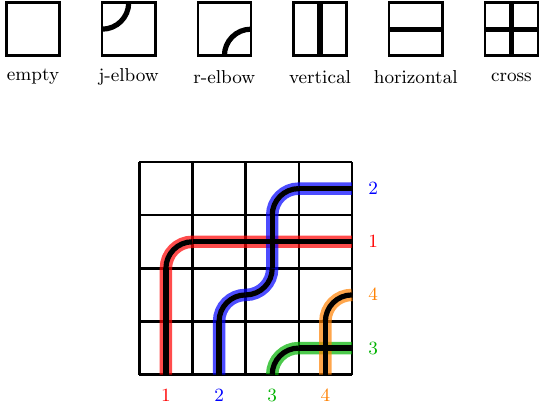}
\caption{\textbf{Top:} The six tile types used in bumpless pipe dreams.
\textbf{Bottom:} A reduced bumpless pipe dream for $n=4$ corresponding to the permutation
$w=2143$.}
\label{fig:bpd_examples_intro}
\end{figure}

\medskip
Schubert polynomials $\mathfrak{S}_w(x_1,\ldots,x_n)$
are defined by attaching variables $x_i$ to the rows of these combinatorial objects.
(More generally, one can assign two sets of variables $x_i, y_j$ in a suitable way to get double Schubert polynomials, see Lascoux-Sch\"utzenberger \cite{LascouxSchutzenberger1985Interpolation}, Macdonald \cite{MacdonaldSchubertBook}.) 
Since here we focus only on the principal specializations,
we simply define
\begin{equation}
	\label{eq:schubert_specialization}
	\begin{split}
		\Upsilon_w&\coloneqq\mathfrak{S}_w(1^n) =\mathfrak{S}_w(\underbrace{1, 1, \ldots, 1}_n)
		\\&\coloneqq \#\{\text{reduced pipe dreams of size } n \text{ with boundary permutation } w\}
		\\&\ssp\ssp= \#\{\text{reduced bumpless pipe dreams of size } n \text{ with boundary permutation } w\}.
	\end{split}
\end{equation}
We recall the full definition and
relevant background results
on Schubert polynomials in
Appendix~\ref{app:schubert_background}.

\subsection{Asymptotics of Schubert polynomial evaluations}
\label{sec:limits-and-conj}
Stanley \cite{stanley2017some} posed the following fundamental question
about the asymptotic behavior of the principal specializations
\eqref{eq:schubert_specialization}:
does the limit
\begin{equation}
	\label{eq:stanley_question}
	\lim_{n\to \infty} \frac{1}{n^2} \log_2 \,\max_{w \in S_n} \Upsilon_w
\end{equation}
exist, and if so, what is its value and for which permutations $w$
is the maximum value of $\Upsilon_w$ achieved?
This question remains open. Morales, Pak, and Panova \cite{MoralesPakPanova2019} established a lower bound of approximately $0.29$ for this limit by using \emph{layered permutations}
(permutations whose diagram consists of consecutive decreasing blocks,
e.g., $32165487$).  An upper bound of approximately $0.37$ follows from a connection with
Alternating Sign Matrices and the six-vertex model. 
We refer to Morales, Panova, Petrov, and Yeliussizov \cite[Section~6]{GrothendieckShenanigans2024} for further discussion.

In~\cite{merzon2016determinantal} Merzon-Smirnov observed
that the maximum of $\Upsilon_w$ is achieved
by layered permutations for values $n\leq 10$ and this formed the basis for the working conjecture:
\begin{conjecture}[\cite{merzon2016determinantal}]
\label{conj:merzon_smirnov}
    The maximal value of $\Upsilon_w$ for $w\in S_n$ is achieved when $w$ is a layered permutation.
\end{conjecture}

This conjecture was exhaustively verified by one of us (DA) for $n \leq 13$ in February 2025. The maximal values for $n\leq 12$ were reported in \cite[sequence A284661]{OEIS}. For $n=13$, the tools developed in the present paper greatly improve the speed of previous computations (and confirm their reliability); this is the first report of the verified maximal value of $\Upsilon_w$ for $w\in S_{13}$ (see \Cref{prop:full_search_13}).

In May 2025, Adam Zsolt Wagner (along with DA and Alejandro Morales) deployed Google DeepMind's FunSearch \cite{romera-paredes2024FunSearch} to seek counterexamples to \Cref{conj:merzon_smirnov}. For $n\leq 16$ the heuristics found by the model did not uncover any counterexamples, providing weak evidence in favor of the conjecture in this range. (For larger $n$, time constraints limited the power of this method.)

At $n=17$, however, using our improved computational methods together with the simple idea of looking at permutations that are ``close'' to layered, we disprove \Cref{conj:merzon_smirnov}.

\begin{theorem}
\label{prop:merzon_smirnov_disproved}
\Cref{conj:merzon_smirnov} is false.
The non-layered permutation
\begin{equation}
\label{eq:counterexample}
w^*=(1,3,2,7,6,5,17,4,16,15,14,13,12,11,10,9,8) \in S_{17},
\end{equation}
obtained from the optimal layered permutation $w(1,2,4,10)$
by a single adjacent transposition $(s_7)$,
satisfies
\[
	\Upsilon_{w^*} = 3{,}272{,}424{,}600{,}397{,}137{,}120{,}000,
\]
exceeding the layered maximum
\[
	\Upsilon_{w(1,2,4,10)} = 3{,}050{,}684{,}475{,}186{,}219{,}300{,}000
\]
by about $7\%$.
The permutation
\begin{equation}
\label{eq:counterexample2}
u^*=(1,3,2,8,6,5,17,4,16,15,14,13,12,11,10,9,7) \in S_{17},
\end{equation}
obtained by two transpositions from $w(1,2,4,10)$,
satisfies
\[
\Upsilon_{u^*}=3{,}528{,}445{,}515{,}842{,}977{,}489{,}500 ,
\]
exceeding $\Upsilon_{w(1,2,4,10)}$ by about $15.6 \%$. 
\end{theorem}
\begin{proof}
Direct computation using the (co)transition formula (Sections~\ref{sec:trans}--\ref{sec:cotrans}).
\end{proof}

While layered permutations fail to achieve the absolute maximum value of $\Upsilon_w$ for large $n$, the general bounds found in Section~\ref{sec:asymp} (\Cref{Prop:bounds} and \Cref{cor:bounds}) imply that the maximal asymptotic behavior appears to be the same, whether one considers all permutations or only layered ones. 

\subsection{Computation and sampling}
\label{sec:intro-compute}

Motivated by these open questions,
we develop computational tools for exploring
principal specializations of Schubert polynomials and the underlying
combinatorial structures. 
We address two main questions:
\begin{itemize}[itemsep=5pt]
	\item How to \textbf{efficiently compute} $\Upsilon_w$ for permutations
	$w\in S_n$ when $n$ is large?
	\item What does a {\bf typical reduced bumpless pipe dream} of size $n$ look like,
	when sampled uniformly at random from all RBPDs
	(regardless of the corresponding permutation~$w$)?
\end{itemize}

The lower bound of about $0.29$ for the (conjectural) limit \eqref{eq:stanley_question} was obtained in \cite{MoralesPakPanova2019} by using an explicit product formula for $\Upsilon_w$
when $w$ is a layered permutation.  This formula enables efficient optimization over layered permutations.
For general permutations, by contrast, even if one can compute $\Upsilon_w$
efficiently for any given $w$, exhaustively searching over all $n!$ permutations
in $S_n$ to find the maximum is infeasible even for moderate $n$ (say, $n>13$). 
This motivated us to study \emph{typical} permutations (in the sense of random
reduced bumpless pipe dreams) and, in particular, compare the values of $\Upsilon_w$ on them with the maximum over layered permutations known from \cite{MoralesPakPanova2019}.

Despite the counterexamples, the simulations and observed \emph{typical} behavior suggest that the asymptotic maximum is still the same as for layered, so we expect that
\[
\limsup_{n\to \infty} \frac{1}{n^2} \log_2 \,\max_{w \in S_n} \Upsilon_w < 0.3
\]
and that the actual limit exists.

\begin{remark}[Typical vs.\ maximal]
The RBPD-typical permutation, whose limiting shape is conjecturally described by the permuton in \Cref{conj:permuton}, does not resemble a layered permutation at large scale, and its $\Upsilon_w$ may be significantly smaller than the maximum.
\end{remark}

For sampling, we first attempted to apply the standard monotone Coupling From The Past (CFTP) algorithm.
However, we show that reduced bumpless pipe dreams do not form a sublattice
of the alternating sign matrix lattice (\Cref{sec:cftp_failure}).
This breaks the monotone coupling required by CFTP: 
the internal rejection scheme needed to enforce reducedness
can cause the two extremal chains to cross, 
invalidating the sandwiching argument. 
Instead, we implement a specialized
Markov chain Monte Carlo (MCMC)
algorithm with macroscopic block updates (\Cref{sec:mcmc_sampler}).

\smallskip

Based on extensive simulations,
we pose {\bf \Cref{conj:permuton} on the limiting permuton} structure
and other properties of typical Schubert permutations and 
{\bf \Cref{conj:delta_boundary} on the limit shape and arctic curves of uniformly random reduced bumpless pipe dreams}
in \Cref{sec:further}. The Schubert and Grothendieck permutons show striking similarities as pictured in Figure~\ref{fig:schubert_grothendieck}. The Grothendieck permuton arises from a non-reduced BPD model and is analyzed rigorously using integrable probability techniques coming from TASEP in~\cite{GrothendieckShenanigans2024}. The requirement in the Schubert model that no pipes cross more than once makes it a nonlocal model and not amenable to integrable techniques at present\footnote{The Schubert model can also be interpreted as a colored vertex model removing the ``long-range interaction'' condition. However, then the number of colors is $n$ and asymptotic analysis cannot be performed.}.

\begin{figure}
    \centering
    \includegraphics[width=0.4\linewidth]{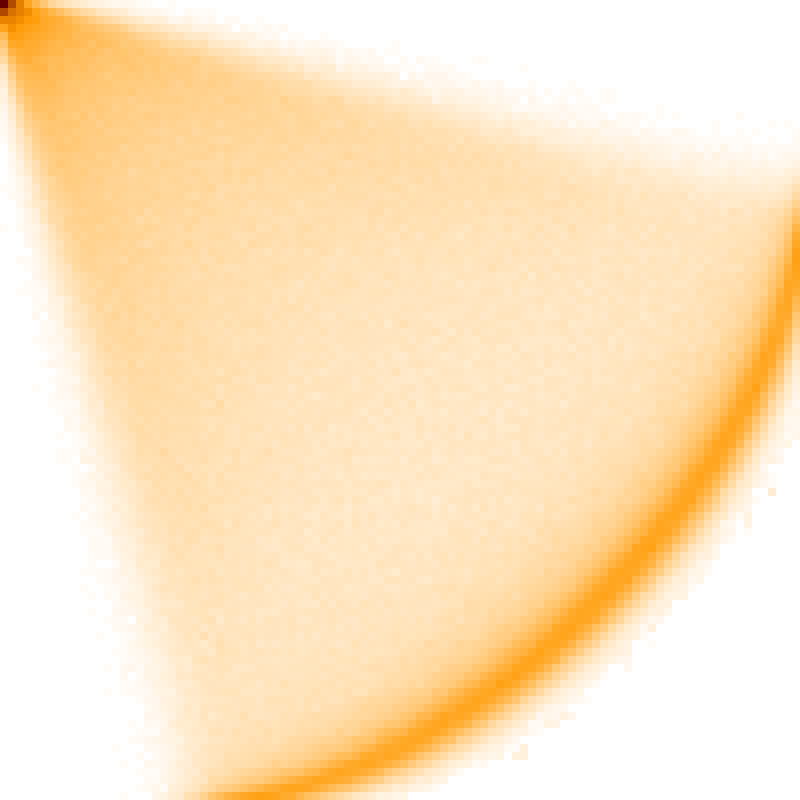} $\qquad$ \includegraphics[width=0.4\textwidth]{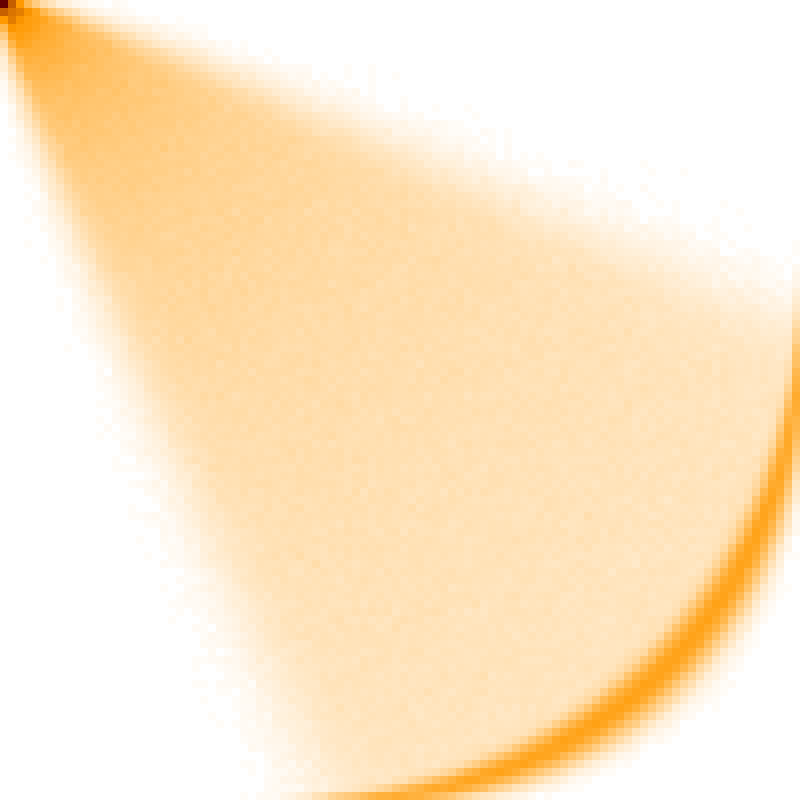}
		\caption{The simulated density histogram for $n=100$ for the Schubert
		(left) and Grothendieck (right) random
		permutations, displaying striking similarities in behavior.
		The permuton limit for the Grothendieck case is proven in \cite[Theorem~1.5]{GrothendieckShenanigans2024}, while the Schubert case is open.}
    \label{fig:schubert_grothendieck}
\end{figure}

\subsection*{Acknowledgements}
We are grateful to the many colleagues in the algebraic combinatorics and integrable probability communities who have expressed interest in this topic.  In particular, we benefited from conversations with Amol Aggarwal, William Fulton, Vadim Gorin, Daoji Huang, Allen Knutson, David Speyer, and Adam Zsolt Wagner; from Igor Pak's rightful suspicion of Conjecture~\ref{conj:merzon_smirnov}; and from fruitful discussions with Richard Stanley about \eqref{eq:stanley_question} and related questions.

A special acknowledgement goes to Alejandro Morales for his earlier work and computations with principal specializations, and for popularizing this topic; as well as to Linus Setiabrata for insightful attempts at proving Conjecture~\ref{conj:merzon_smirnov}.  We heartily thank both of them for many enjoyable and inspiring conversations.

\subsection*{Note on AI tools}

We used Claude Code as a coding assistant for optimizing the Schubert evaluation and sampling code, and then subjected our code to thorough testing and verification across multiple implementations (C++ and Julia).
After we found the first permutation
$w^*$ \eqref{eq:counterexample} violating \Cref{conj:merzon_smirnov}, the second example $u^*$ \eqref{eq:counterexample2} was obtained with the help of ChatGPT-5.4 Pro, in the process of verifying the computation of $\Upsilon_{w^*}$.

\section{Computation of Schubert specializations}
\label{sec:computation}

We use three recurrence formulas for computing
$\Upsilon_w$.
In this section, we state these formulas and discuss their implementation details.
Performance benchmarks and comparisons of the implementations are presented in \Cref{sec:comparison}.

\subsection{Notation for permutations}
\label{sec:notation}

We fix some standard notation:
$\ell(w)$ denotes the number of inversions of $w$,
$s_i = (i,i+1)$ are adjacent transpositions, $t_{i,j}=(i,j)$ are not necessarily adjacent transpositions,
$\Des(w) = \{i : w(i) > w(i+1)\}$ is the descent set,
$e = (1,2,\ldots,n)$ is the identity, and
$w_0 = (n,n-1,\ldots,1)$ is the longest permutation.
We use right multiplication on one-line notation:
$w\cdot(a,b)$ is obtained from $w$ by swapping the entries in positions $a,b$.
A word $(a_1,\ldots,a_k)$ with $1 \le a_i < n$
is called a \emph{reduced word} for a permutation $w \in S_n$
if $w = s_{a_1} s_{a_2} \cdots s_{a_k}$ and $k = \ell(w)$
(that is, this expression for $w$ in elementary transpositions
is of minimal possible length).

The permutation matrix associated to $w$ has $1$'s in positions $(i,w(i))$ and $0$'s elsewhere.  Often we indicate the positions of $1$'s by a dot, leaving the $0$'s blank.

\subsection{Descent formula}
\label{sec:descent}

The descent recurrence formula \cite{MacdonaldSchubertBook}, \cite{FominStanley1994}, \cite{BilleyHolroydYoung2019} expresses
$\Upsilon_w$ in terms of values at permutations of smaller length,
where the sum is over the descents:

\begin{theorem}[Descent formula]
	\label{thm:descent}
	For $w \in S_n$ with $w \ne e$, we have:
	\begin{equation}
		\label{eq:descent}
		\Upsilon_w = \sum_{i \in \Des(w)}
		\frac{i}{\ell(w)} \cdot \Upsilon_{w \cdot s_i} = \sum_{i \in \Des(w^{-1})} \frac{i}{\ell(w)} \cdot \Upsilon_{s_i \cdot w}.
	\end{equation}
	The base case is $\Upsilon_e = 1$.
\end{theorem}

Note that each permutation $w \cdot s_i$ occurring on the RHS of \eqref{eq:descent} has length $\ell(w)-1$, so the recurrence terminates at $\Upsilon_e=1$.

\begin{proof}[Proof of \Cref{thm:descent}]
	From \cite[(6.11)]{MacdonaldSchubertBook}, we have
	\begin{equation}
	\label{eq:macdonald_formula}
		\Upsilon_w = \frac{1}{\ell(w)!} \sum_{(a_1,\ldots,a_{\ell(w)}) \in R(w)} a_1 a_2 \cdots a_{\ell(w)},
	\end{equation}
	where $R(w)$ is the set of reduced words for $w$.
	Splitting the sum according to the descents of $w$ (or $w^{-1}$ for the second sum) and using
	\eqref{eq:macdonald_formula} for the resulting shorter permutations
	$w \cdot s_i$ yields the desired recurrence.
\end{proof}

Let us discuss the implementation of the descent formula of \Cref{thm:descent}.
	\begin{enumerate}[{\bf 1.}]
		\item \textbf{Permutation encoding}:
		Permutations are encoded as 128-bit integers using 5 bits per element,
		supporting $n \le 25$, which enables fast hashing and efficient memory usage.
		Adjacent transposition $w \mapsto w \cdot s_i$ is performed via bitwise operations
		directly on the packed representation.
		For enumeration over $S_n$ (\Cref{sec:unified_max}), we use a more compact
		64-bit encoding (4 bits per element) for $n \le 16$, halving memory usage.

		\item \textbf{Preprocessing.}
		Before running the recurrence, we strip trailing fixed points of $w$
		(suffix positions where $w(k)=k$), since this does not change $\Upsilon_w$,
		to reduce the effective size of the problem, if possible.

		\item \textbf{Memoization and length decrement.}
		We store computed values in a hash table keyed by the packed permutation,
		with a hard cap on the number of entries ($3 \cdot 2^{26} \approx 192\text{M}$ for double,
	$2^{27} \approx 128\text{M}$ for exact arithmetic)
	to keep the table cache-friendly.
	The length-decreasing nature of the recursion lets us pass $\ell(w)$
	as a parameter and decrement it at each step,
	avoiding the $O(n^2)$ cost of recomputing length from scratch.

	\item \textbf{BFS sort-reduce evaluation.}
	We evaluate the descent recurrence level-by-level
	using a breadth-first search (BFS) sort-reduce algorithm.
	Starting from the target~$w$ at level~$\ell(w)$,
	we process all permutations at a given length simultaneously.
	For each permutation~$v$ at the current level and each descent $i \in \Des(v)$,
	we emit the child $v \cdot s_i$ (at length $\ell(v)-1$)
	with its weighted contribution $\frac{i}{\ell(v)} \cdot \Upsilon_v$.
	Distinct parents may produce the same child.
	We \emph{sort} the children by their packed key
	and \emph{reduce} by summing contributions for each distinct key,
	yielding $\Upsilon$ for every permutation at the new level.
	The previous level is then discarded.

	\item \textbf{Arithmetic precision.}
	Formula \eqref{eq:descent} involves division by $\ell(w)$,
	requiring either rational or floating-point arithmetic.
	We provide two implementations:
	\begin{enumerate}[(a)]
		\item \emph{Double-precision arithmetic} (64-bit): Provides approximately
		15 significant digits, which becomes imprecise for large values of $\Upsilon_w$.
		\item \emph{Rational arithmetic} (exact): Uses
		a custom implementation of rational numbers with
		128-bit integers for
		numerator and denominator and GCD reduction after each operation.
		Produces exact integer results (when they fit into 128 bits) at slower speed.
	\end{enumerate}
	Since $\Upsilon_w$ is always an integer (counting reduced pipe dreams),
	the rational implementation serves as ground truth for validation.
\end{enumerate}

\subsection{Transition formula}
\label{sec:trans}

The transition formula of Lascoux-Sch\"utzenberger \cite{LascouxSchutzenberger1985LR} (and its refinement by Fan-Guo-Sun \cite{FanGuoSun2018Bumpless}) 
can be specialized to $\Upsilon_w$ as follows.
Given $w = w_1\,w_2\,\cdots w_n$, let $r$ be the largest index such that
$w_r$ appears as the ``$3$'' in a $132$ pattern, i.e.,
there exist $i<r<s$ with $w_i<w_s<w_r$.
If no such index exists, then $w$ is dominant and $\Upsilon_w=1$
\cite[(4.7)]{MacdonaldSchubertBook}.
Otherwise, we have the following recurrence:
\begin{theorem}[Transition formula {\cite{LascouxSchutzenberger1985LR}}]
\label{thm:trans}
For non-dominant $w \in S_n$, let $r$ be as above, and let $s>r$ be
the largest index such that $w_s<w_r$ and there exists $i<r$ with
$w_i<w_s$. Then
\begin{equation}
\label{eq:transition}
\Upsilon_w = \Upsilon_{v} + \sum_{i<r\,:\,\ell(v\cdot(i,r))=\ell(w)} \Upsilon_{v\cdot(i,r)},
\end{equation}
where $v=w\cdot(r,s)$.
\end{theorem}
By the choice of $r$ and $s$, we have $\ell(v)=\ell(w)-1$.

\medskip
The implementation of \Cref{thm:trans} uses the same packed permutation
representation as in \Cref{sec:descent}, and is based on iterative
depth-first search (DFS) with memoization.
	\begin{enumerate}[{\bf 1.}]
		\item \textbf{Packed representation and memoization.}
		States are stored as 128-bit packed permutations (5 bits per entry, $n \le 25$)
		and memoized in hash tables keyed by the packed code.
		As in \Cref{sec:descent}, we first strip trailing fixed points of $w$,
		and use the same memoization hard caps.

	\item \textbf{Transition index computation and base case.}
	For each state, we compute $(r,s)$ directly on the packed representation.
	If no such pair exists (equivalently, the permutation is dominant), we apply
	the base case $\Upsilon_w=1$.

	\item \textbf{Child generation and depth-first search state.}
	For non-dominant $w$, we first form $v=w\cdot(r,s)$ and then generate all
	terms $v\cdot(i,r)$ satisfying
	$\ell(v\cdot(i,r))=\ell(w)$.
	We then follow one branch of this recursion tree to a
	base case before backtracking to process remaining children.
	The stack frame stores this child list explicitly; the first child has
	length $\ell-1$, while the remaining children stay at length $\ell$.
	Because of these same-level dependencies, this implementation uses
	depth-first search (DFS) rather than a level-by-level BFS variant.

	\item \textbf{Arithmetic precision.}
	We provide two precision versions: one with double precision, and another one with exact
	integer arithmetic based on GMP's \texttt{mpz\_class}.
	Since \eqref{eq:transition} uses only addition,
	no rational division is required.
\end{enumerate}

\subsection{Cotransition formula}
\label{sec:cotrans}

The cotransition formula of Knutson \cite{Knutson2019cotransition} provides
yet
another recurrence that uses only integer addition. Recall that
$w_0$ denotes the longest permutation.

	\begin{theorem}[Cotransition formula {\cite{Knutson2019cotransition}}]
		\label{thm:cotrans}
		For $w \in S_n$ with $w \ne w_0$, let us denote
		\begin{equation}
	\label{eq:cotrans_index}
		i\coloneqq \min\{j : j + w(j) \le n\}.
	\end{equation}
	Then we have:
		\begin{equation}
			\label{eq:cotrans}
			\Upsilon_w = \sum_{v \gtrdot w,\,  v(i) \ne w(i)} \Upsilon_v.
		\end{equation}
		Here $v \gtrdot w$ means that $v$ covers $w$ in Bruhat order
		(see \Cref{lemma:bruhat_cover} below).
		The base case is $\Upsilon_{w_0} = 1$.
	\end{theorem}

\begin{remark}
	The cotransition formula \eqref{eq:cotrans} proceeds
	\emph{toward} the longest permutation $w_0$, increasing the length
	at each step.
	Like the transition formula, \eqref{eq:cotrans} also
	involves only addition, and thus can be implemented using exact integer arithmetic.
\end{remark}

\begin{remark}
	From \eqref{eq:cotrans} it is not
	hard to see that there exists
	$w\in S_n$ with the maximal $\Upsilon_w$
	such that 
	$w(1)=1$.
\end{remark}

The following well-known characterization of the Bruhat order is useful
for implementation:
\begin{lemma}[{\cite[Lemma~2.1.4]{BjornerBrenti2005}}]
	\label{lemma:bruhat_cover}
	A permutation $v$ covers $w$ in Bruhat order if and only if
	$v = w \cdot (a, b)$ for some $a < b$ such that
	$w(a)<w(b)$, and there is no $k$ with $a < k < b$ and $w(a) < w(k) < w(b)$.
\end{lemma}

Let us discuss the implementation details of the cotransition formula of \Cref{thm:cotrans}.
\begin{enumerate}[{\bf 1.}]
	\item \textbf{Permutation encoding.}
	As in \Cref{sec:descent,sec:trans}, we first strip trailing fixed points of $w$
	before evaluation.
	Permutations are then handled in the same packed form,
	64-bit for $n \le 16$
	and 128-bit for $n \le 25$.

	\item \textbf{BFS sort-reduce evaluation.}
	The cotransition evaluator runs level-by-level using
	breadth-first search (BFS). 
	At each level, for each permutation~$w$ on the current frontier
	with value~$\Upsilon_w$,
	we enumerate all Bruhat covers $v \gtrdot w$ satisfying $v(i) \ne w(i)$
	and emit the pair $(v, \Upsilon_w)$.
	Since distinct permutations at the same level may share a common cover~$v$,
	we apply sort-reduce as in \Cref{sec:descent}:
	sort emitted pairs by packed key and sum the values for each distinct~$v$.
	This computes $\Upsilon_w = \sum_v \Upsilon_v$ per \eqref{eq:cotrans}
	for all permutations at the new level.
	This collapses repeated states and keeps memory usage bounded.\footnote{An older implementation used DFS with memoization but it turned
	out to be much slower, as reported in \Cref{rmk:long_computation} below.}

	\item \textbf{Bruhat cover enumeration.}
	For each permutation $w$, we enumerate all Bruhat covers $v \gtrdot w$
	as in \Cref{lemma:bruhat_cover}.
	We filter to those covers where $v(i) \ne w(i)$ at the cotransition index $i$,
	see \eqref{eq:cotrans_index}.
	We store the cover pairs in a fixed-size array (not heap-allocated),
	eliminating memory allocation overhead in the inner loop.

	\item \textbf{Arithmetic precision.}
	Since formula \eqref{eq:cotrans} involves only addition,
	exact integer arithmetic is natural. We provide two implementations:
	\begin{enumerate}[(a)]
		\item \emph{Exact arithmetic} uses the GMP library (\texttt{mpz\_class})
		to compute exact integer values regardless of magnitude.
		This is essential for large $n$, where $\Upsilon_w$ can exceed 50 digits.
		\item \emph{Double-precision} (64-bit) is typically faster,
		but may become imprecise for large $\Upsilon_w$.
	\end{enumerate}
\end{enumerate}

\begin{remark}
   A folklore rule of thumb holds that the transition formula is the most efficient means of computing Schubert polynomials. (See discussion and debate in \cite[\S1.2]{monical2022reduced}.)  As we will see in the next section, this expectation is close to true in our simplified setting of principal evaluations, especially as $n$ grows.  The similar cotransition formula performs equally well or better for smaller $n$.
\end{remark}

\section{Performance comparison}
\label{sec:comparison}

We present a comparison of the performance of the descent, transition, and cotransition formulas from \Cref{sec:computation}.  
Our comparisons are anchored by known closed formulas for {\em layered permutations} and, in particular, the 
permutations maximizing $\Upsilon_w$ over the layered ones
\cite{MoralesPakPanova2019} (which are \emph{not} the absolute maximizers over all permutations, 
see \Cref{prop:merzon_smirnov_disproved} for a counterexample).
Most computations were performed on a laptop with Apple M2 Pro chip and 16GB~RAM,
and some of the heavier computations (where indicated)
used a more powerful desktop computer with an AMD Ryzen 9 processor with
64GB~RAM.

\subsection{Layered permutations}
\label{sec:layered_benchmark}

For a composition $(b_1,\ldots,b_k)$ of $n$ we
recursively define the \emph{layered permutation}
$w(b_1,\ldots,b_k):=
\bigl(w(b_1,\ldots,b_{k-1}), n, n-1,\ldots,n-b_k+1\bigr)$, where $w(b) = \bigl(b,b-1,\ldots,1\bigr)$.
In other words, the layered permutation $w(b_1,\ldots,b_k)$
is obtained by splitting $1,\ldots,n$ into blocks of sizes $b_1,b_2,\ldots,b_k$
and reversing each block.

It is straightforward to compute the length of a layered permutation:
\begin{lemma}
	\label{lem:layered_length}
	The length of a layered permutation $w(b_1,\ldots,b_k)$ is
	equal to
	$\sum_{i=1}^{k} \binom{b_i}{2}$.
\end{lemma}

From \cite{MoralesPakPanova2019}, we explicitly know the permutations on which
$\Upsilon_w$ achieves the maximum among the layered permutations.  Let $w^*(n)$ denote this permutation, i.e., $w^*(n)$ is the layered permutation maximizing $\Upsilon_w$ over all layered permutations in $S_n$.
\begin{lemma}
	\label{lem:layered_length_ratio}
    We have
	\begin{equation}
		\label{eq:layered_length_ratio}
		\lim_{n\to\infty} \frac{\ell(w^*(n))}{\binom{n}{2}}
		= \frac{1-\alpha}{1+\alpha} \approx 0.3955,
	\end{equation}
	where $\alpha \approx 0.4331818312$ is the constant
	from \cite{MoralesPakPanova2019}.
\end{lemma}
\begin{proof}
	By \cite[Theorem~1.1]{MoralesPakPanova2019},
	the optimal block sizes satisfy
	$b_i \sim \alpha^{i-1}(1-\alpha)\hspace{1pt} n$ as $n\to\infty$.
	By \Cref{lem:layered_length},
	\[
		\frac{\ell(w^*(n))}{\binom{n}{2}}
		\sim \frac{\sum_i b_i^2}{n^2}
		= (1-\alpha)^2 \sum_{i=0}^{\infty} \alpha^{2i}
		= \frac{(1-\alpha)^2}{1-\alpha^2}
		= \frac{1-\alpha}{1+\alpha}. \qedhere
	\]
\end{proof}

As the first benchmark, we compute $\Upsilon_{w^*(n)}$ for the maximizing layered permutations, for $n$ up to $17$, using all three formulas
with double-precision and exact arithmetic implementations.
\Cref{tab:layered_benchmark} summarizes the timing results.

\begin{table}[htbp]
	\centering
	\caption{Timing comparison for maximal layered permutations from \cite{MoralesPakPanova2019}.
	Times are in seconds. The ``Layers'' column shows the block structure of an optimal layered
	permutation (for example, for $n=8$ the permutation is $w(1,2,5)=13287654$). Asterisks indicate results from the range where integer accuracy is lost. 
	The descent rational variant exceeded the 180s timeout at $n \ge 16$.
	The row $n=300$ shows the largest value computed in \cite{MoralesPakPanova2019} for comparison
	(no timing data).}
	\label{tab:layered_benchmark}
	\small
	\begin{tabular}{cccccccccc}
		\toprule
		$n$ & Layers & $\ell(w)$ & $\frac{\log_2 \Upsilon_w}{n^2}$ &
		\multicolumn{2}{c}{Descent} & \multicolumn{2}{c}{Cotransition} & \multicolumn{2}{c}{Transition} \\
		\cmidrule(lr){5-6} \cmidrule(lr){7-8} \cmidrule(lr){9-10}
		& & & & double & rational & double & exact & double & exact \\
		\midrule
		8 & $(1,2,5)$ & 11 & 0.206 & 0.0001 & 0.0002 & 0.0001 & 0.0002 & 0.101 & 0.075 \\
		9 & $(1,2,6)$ & 16 & 0.214 & 0.0002 & 0.0017 & 0.0005 & 0.0005 & 0.085 & 0.088 \\
		10 & $(1,3,6)$ & 18 & 0.221 & 0.0006 & 0.0059 & 0.0009 & 0.0013 & 0.112 & 0.108 \\
		11 & $(1,3,7)$ & 24 & 0.227 & 0.0069 & 0.0320 & 0.0040 & 0.0084 & 0.094 & 0.104 \\
		12 & $(1,3,8)$ & 31 & 0.230 & 0.0373 & 0.2770 & 0.0168 & 0.0255 & 0.147 & 0.127 \\
		13 & $(1,1,3,8)$ & 31 & 0.234 & 0.0444 & 0.2823 & 0.0370 & 0.0582 & 0.185 & 0.139 \\
		14 & $(1,1,4,8)$ & 34 & 0.237 & 0.1491 & 1.1652 & 0.0784 & 0.1328 & 0.287 & 0.339 \\
		15 & $(1,1,4,9)$ & 42 & 0.242 & $1.4425^*$ & 14.5971 & 0.4032 & 0.7926 & 0.902 & 1.108 \\
		16 & $(1,1,4,10)$ & 51 & 0.244 & $20.5393^*$ & ${>}180$ & $2.4683^*$ & 5.6003 & $3.367^*$ & 4.316 \\
		17 & $(1,2,4,10)$ & 52 & 0.247 & $93.5795^*$ & ${>}180$ & $9.7029^*$ & 17.6501 & $9.779^*$ & 12.966 \\
		\midrule
		300 & $(1,2,6,14,32,74,171)$ & 17839 & 0.290 & --- & --- & --- & --- & --- & --- \\
		\bottomrule
	\end{tabular}
\end{table}

For layered permutations, cotransition is the fastest method overall in double precision.
The transition formula has larger constant overhead for small $n$
($\approx 0.1$ seconds already for $n=8$--$11$), but scales comparably to cotransition for larger $n$.
In exact arithmetic, transition is already faster than cotransition at $n=16,17$.
The descent formula remains competitive only at very small $n$ and then grows much more steeply;
its exact (rational) variant hits the 180s timeout by $n=16$.

\begin{remark}[Precision]
	\label{rmk:precision}
	At $n=15$, the maximal value of $\Upsilon_w$ is $\approx 2.3 \times 10^{16}$, which exceeds $2^{53} \approx 9 \times 10^{15}$,
	the threshold beyond which double-precision floating-point arithmetic loses integer precision.
	The descent formula with double precision yields $23{,}399{,}330{,}089{,}073{,}392$,
	while exact arithmetic gives $23{,}399{,}330{,}089{,}073{,}400$ --- a discrepancy of 8 in the units digit.
	The cotransition and transition formulas with double precision still produce the correct integer value at $n=15$.
	At $n \ge 16$, both cotransition and transition in double precision lose integer accuracy.
\end{remark}

\subsection{Random permutations from the RBPD sampler}
\label{sec:random_benchmark}

We also test the performance
on typical permutations obtained from the reduced
bumpless pipe dream (RBPD)
sampler described in \Cref{sec:mcmc_sampler} below.
We pick $n=15$ and sample 200 permutations.
\Cref{tab:random_benchmark} summarizes the timing.

\begin{table}[htbp]
	\centering
	\caption{Performance statistics for 200 random permutations from the RBPD sampler at $n=15$.}
	\label{tab:random_benchmark}
	\small
	\begin{tabular}{lcccccc}
		\toprule
		& \multicolumn{2}{c}{Descent} & \multicolumn{2}{c}{Cotransition} & \multicolumn{2}{c}{Transition} \\
		\cmidrule(lr){2-3} \cmidrule(lr){4-5} \cmidrule(lr){6-7}
		Statistic & double & rational & double & exact & double & exact \\
		\midrule
		Mean time (s) & 0.213 & 2.262 & 0.017 & 0.030 & 0.314 & 0.365 \\
		Median time (s) & 0.105 & 0.974 & 0.012 & 0.021 & 0.271 & 0.312 \\
		Max time (s) & 2.961 & 34.007 & 0.099 & 0.189 & 1.242 & 1.601 \\
		Total time (s) & 42.7 & 452.3 & 3.5 & 5.9 & 62.8 & 73.1 \\
		\bottomrule
	\end{tabular}
\end{table}

The cotransition formula is by far the fastest on these RBPD-typical permutations.

\begin{remark}
	\label{rmk:long_computation}
	For larger $n$, the computation time grows substantially.
	For example, the RBPD-typical permutation
	\begin{equation*}
		w = (1,3,11,2,8,4,13,7,21,9,6,19,17,18,5,16,15,14,20,10,12) \in S_{21}
	\end{equation*}
	with $\ell(w) = 67$ requires about 29 seconds using the BFS-mode exact-precision cotransition formula
	(the DFS mode with memoization takes several dozen minutes, and
	the descent formula does not terminate within hours).
	For
	a RBPD-typical permutation
	\begin{equation*}
		w=(2,1,4,13,10,3,7,19,6,23,22,12,17,21,5,16,20,11,18,14,15,9,8)\in S_{23}
	\end{equation*}
	with $\ell(w) = 98$, the BFS-mode cotransition formula completes in about 9 minutes,
	while the DFS mode did not finish within a week.
\end{remark}

\begin{remark}
	\label{rmk:descent_better}
	While the cotransition formula is much faster on RBPD-typical permutations,
	the descent formula wins decisively on permutations close to the identity.
	For example, consider
	\begin{equation*}
		w = (3,4,5,1,2,8,9,10,6,7,13,14,15,11,12,18,19,20,16,17,22,21,23,24,25) \in S_{25}
	\end{equation*}
	with $\ell(w) = 25$.
	After stripping trailing fixed points ($n$ reduced to $22$),
	the descent formula computes $\Upsilon_w$
	in about 0.005 seconds,
	while the cotransition formula takes about 0.03 seconds.
\end{remark}
%
%

\subsection{Search for maximum}
\label{sec:unified_max}

We search for
the permutation maximizing $\Upsilon_w$
over all $w \in S_n$ and not just over the layered ones.

In the \emph{full search for maximum} up to $n=13$, we start two parallel threads:
the descent thread computes $\Upsilon_w$ starting from the identity and proceeding upward in length,
while the cotransition thread starts from $w_0$ and proceeds downward.
Each thread uses the sort-reduce approach described in \Cref{sec:descent,sec:cotrans}.
This keeps only the current level in memory.
The two threads use \emph{dynamic meeting}: each checks whether it has crossed
the other thread's current level, and stops when the levels cross.
This approach balances the workload between threads,
as the descent and cotransition frontiers grow at different rates.
Both threads track the maximum $\Upsilon_w$ encountered at each level,
guaranteeing that the global maximum is found.
We run the full search up to $n=12$, where it completes in about 1~minute.
For $n=13$, this implementation runs out of the 16GB RAM available,
but was completed in about 15 minutes on a machine with 64GB RAM.
During some of the steps, it had exhausted all of the available RAM
and used the Linux swap space.

We have the following experimental result:

\begin{proposition}
\label{prop:full_search_13}
For $n \le 13$, the permutation maximizing $\Upsilon_w$ over all $w \in S_n$
is the same as the one maximizing over layered permutations,
as given in \cite{MoralesPakPanova2019}.
\end{proposition}

The prior state of the art was exhaustive verification for $n \le 10$,
as reported in \cite{merzon2016determinantal}.
The cases $n = 11, 12$ were subsequently verified and reported in \cite{OEIS} (sequence A284661),
and the present work provides the first published record of exhaustive search
through $n = 13$.
For $14 \le n \le 16$, exhaustive search is infeasible,
but no counterexample was found within Cayley distance~$4$
of the optimal layered permutation.
At $n = 17$, however, a permutation at Cayley distance~$1$
from the optimal layered permutation exceeds the layered maximum
(\Cref{prop:merzon_smirnov_disproved}),
along with similar counterexamples for $18\leq n\leq 20$
(\Cref{sec:merzon_smirnov_further}).

\section{Sampling: failure of CFTP with local flips}
\label{sec:cftp_failure}

In this section, we analyze the state space of reduced bumpless pipe dreams
(RBPDs) under local tile moves. A natural approach
for exact sampling of uniformly random RBPDs
is to embed them within the larger space of alternating sign matrices (ASMs)
and use Coupling From The Past (CFTP)
\cite{ProppWilsonCP}. We explain why this naive approach fails.

\subsection{Rothe diagram and Rothe BPD}
\label{sec:Rothe}

The \emph{Rothe diagram} of $w\in S_n$ is
\begin{equation}
	\label{eq:rothe}
	D(w) \coloneqq \bigl\{(i,j)\in [n]\times [n] : w(i) > j \text{ and } w^{-1}(j) > i\bigr\}.
\end{equation}
Equivalently, $D(w)$ is obtained by placing a dot at $(i,w(i))$ for each $i$
and removing all cells weakly right of or below each dot (in the same row and column, respectively). The number of cells is equal to the length: $|D(w)| = \ell(w)$. See \Cref{fig:rothe_diagram} for an example.

\begin{figure}[htpb]
\centering
	\begin{tikzpicture}[scale=0.8]
	\def\n{4}

	\draw[step=1,black] (0,0) grid (\n,\n);

	\fill[blue!15] (0,3) rectangle (1,4);
	\fill[blue!15] (1,3) rectangle (2,4);
	\fill[blue!15] (1,1) rectangle (2,2);

	\fill (2.5,3.5) circle (3.5pt); 
	\fill (0.5,2.5) circle (3.5pt); 
	\fill (3.5,1.5) circle (3.5pt); 
	\fill (1.5,0.5) circle (3.5pt); 
	\draw[thick] (.5,0)--++(0,2.5)--++(3.5,0);
	\draw[thick] (1.5,0)--++(0,.5)--++(2.5,0);
	\draw[thick] (2.5,0)--++(0,3.5)--++(1.5,0);
	\draw[thick] (3.5,0)--++(0,1.5)--++(.5,0);

	\foreach \j in {1,...,4} {
		\node[font=\small] at ({\j-0.5},-0.45) {\j};
	}
	\foreach \i in {1,...,4} {
		\node[font=\small] at (-0.45,{4-\i+0.5}) {\i};
	}
	\end{tikzpicture}
	\caption{Rothe diagram $D(w)$ for $w=(3,1,4,2)$.
	Dots mark positions $(i,w(i))$; shaded cells form $D(w)$.
	(Here $|D(w)| = 3 = \ell(w)$.)}
\label{fig:rothe_diagram}
\end{figure}
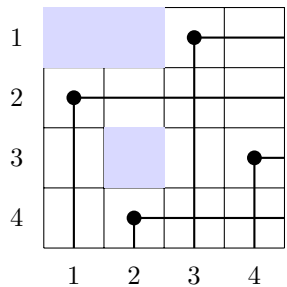

Interpreting the shaded cells of the diagram $D(w)$ as empty, and the rays deleting the boxes as pipes, we obtain a reduced bumpless pipe dream $\mathbf{b}_w$ for $w$, called the \emph{Rothe BPD}.  This is an RBPD canonically associated to $w$.

\subsection{Height functions and local flips}
\label{sec:local_flips}

Recall that a BPD uses the six tile types shown in
\Cref{fig:bpd_examples_intro}. Disregarding the reducedness constraint, the set
of all BPDs of size $n$ is in bijection with configurations of the six-vertex
model with domain wall boundary conditions, or equivalently, ASMs of size $n$
\cite{kuperberg1996another}, \cite{Bressoud1999}.
This space is equipped with a natural lattice structure governed by height
functions:

\begin{definition}[Height function and partial order]
\label{def:height}
For an $n \times n$ BPD $D$, the \emph{height function} $h_D(i,j) \in
\mathbb{Z}_{\ge 0}$ is defined on the dual-lattice vertices $0 \le i,j \le n$
so that six-vertex lines are level lines of $h_D$.

We define a partial order on all
BPDs by $D \le D'$ if and only if $h_D(v) \le h_{D'}(v)$ pointwise for all
vertices $v$ of the dual square lattice.
Under this order, ASMs form a distributive lattice with a unique
minimum $\mathbf{b}_{w_0}$ and
maximum $\mathbf{b}_{\mathrm{id}}$ (the Rothe BPDs for the
longest and the identity
permutation, respectively).
We refer to
\Cref{fig:height_function} for an illustration.
\end{definition}

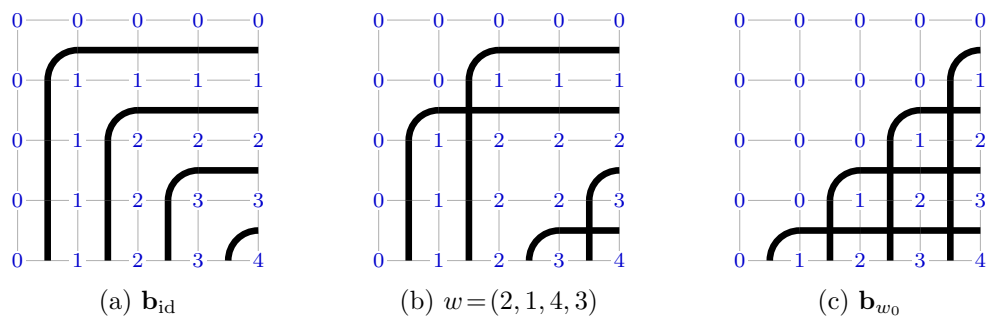
\begin{figure}[htbp]
\centering
\begin{tikzpicture}[scale=0.8]
  \begin{scope}[shift={(0,0)}]
  \draw[thin, gray!60] (0,0) grid (4,4);
  \brelbow{0}{3} \bhoriz{1}{3} \bhoriz{2}{3} \bhoriz{3}{3}
  \bvert{0}{2} \brelbow{1}{2} \bhoriz{2}{2} \bhoriz{3}{2}
  \bvert{0}{1} \bvert{1}{1} \brelbow{2}{1} \bhoriz{3}{1}
  \bvert{0}{0} \bvert{1}{0} \bvert{2}{0} \brelbow{3}{0}
  \foreach \x/\y/\h in {
    0/4/0, 1/4/0, 2/4/0, 3/4/0, 4/4/0,
    0/3/0, 1/3/1, 2/3/1, 3/3/1, 4/3/1,
    0/2/0, 1/2/1, 2/2/2, 3/2/2, 4/2/2,
    0/1/0, 1/1/1, 2/1/2, 3/1/3, 4/1/3,
    0/0/0, 1/0/1, 2/0/2, 3/0/3, 4/0/4
  } {
    \node[fill=white, inner sep=1pt, font=\scriptsize, text=blue!80!black] at (\x,\y) {\h};
  }
  \node[below, font=\small] at (2,-0.3) {(a) $\mathbf{b}_{\mathrm{id}}$};
  \end{scope}
  \begin{scope}[shift={(6,0)}]
  \draw[thin, gray!60] (0,0) grid (4,4);
  \brelbow{1}{3} \bhoriz{2}{3} \bhoriz{3}{3}
  \brelbow{0}{2} \bcross{1}{2} \bhoriz{2}{2} \bhoriz{3}{2}
  \bvert{0}{1} \bvert{1}{1} \brelbow{3}{1}
  \bvert{0}{0} \bvert{1}{0} \brelbow{2}{0} \bcross{3}{0}
  \foreach \x/\y/\h in {
    0/4/0, 1/4/0, 2/4/0, 3/4/0, 4/4/0,
    0/3/0, 1/3/0, 2/3/1, 3/3/1, 4/3/1,
    0/2/0, 1/2/1, 2/2/2, 3/2/2, 4/2/2,
    0/1/0, 1/1/1, 2/1/2, 3/1/2, 4/1/3,
    0/0/0, 1/0/1, 2/0/2, 3/0/3, 4/0/4
  } {
    \node[fill=white, inner sep=1pt, font=\scriptsize, text=blue!80!black] at (\x,\y) {\h};
  }
  \node[below, font=\small] at (2,-0.3) {(b) $w\!=\!(2,1,4,3)$};
  \end{scope}
  \begin{scope}[shift={(12,0)}]
  \draw[thin, gray!60] (0,0) grid (4,4);
  \brelbow{3}{3}
  \brelbow{2}{2} \bcross{3}{2}
  \brelbow{1}{1} \bcross{2}{1} \bcross{3}{1}
  \brelbow{0}{0} \bcross{1}{0} \bcross{2}{0} \bcross{3}{0}
  \foreach \x/\y/\h in {
    0/4/0, 1/4/0, 2/4/0, 3/4/0, 4/4/0,
    0/3/0, 1/3/0, 2/3/0, 3/3/0, 4/3/1,
    0/2/0, 1/2/0, 2/2/0, 3/2/1, 4/2/2,
    0/1/0, 1/1/0, 2/1/1, 3/1/2, 4/1/3,
    0/0/0, 1/0/1, 2/0/2, 3/0/3, 4/0/4
  } {
    \node[fill=white, inner sep=1pt, font=\scriptsize, text=blue!80!black] at (\x,\y) {\h};
  }
  \node[below, font=\small] at (2,-0.3) {(c) $\mathbf{b}_{w_0}$};
  \end{scope}
\end{tikzpicture}
\caption{Height functions on $n=4$ BPDs. Values of $h_D(i,j)$ appear at
dual-lattice vertices; six-vertex lines are the level lines of $h_D$.}
\label{fig:height_function}
\end{figure}

\begin{figure}[htbp]
\centering
\begin{tikzpicture}[scale=0.5]
  \node[left,font=\small] at (-0.3,11) {(a)};
  \node[left,font=\small] at (-0.3,6) {(b)};
  \node[left,font=\small] at (-0.3,1) {(c)};
  \begin{scope}[shift={(0,10)}]
    \draw[thin,gray!60] (0,0) grid (2,2);
    \brelbow{1}{1} \bhoriz{0}{0} \bjelbow{1}{0}
    \node at (2.5,1) {$\leftrightarrow$};
    \begin{scope}[shift={(3,0)}]
      \draw[thin,gray!60] (0,0) grid (2,2);
      \brelbow{0}{1} \bhoriz{1}{1} \bjelbow{0}{0}
    \end{scope}
  \end{scope}
  \begin{scope}[shift={(7,10)}]
    \draw[thin,gray!60] (0,0) grid (2,2);
    \brelbow{1}{1} \brelbow{0}{0} \bjelbow{1}{0}
    \node at (2.5,1) {$\leftrightarrow$};
    \begin{scope}[shift={(3,0)}]
      \draw[thin,gray!60] (0,0) grid (2,2);
      \brelbow{0}{1} \bhoriz{1}{1} \bvert{0}{0}
    \end{scope}
  \end{scope}
  \begin{scope}[shift={(14,10)}]
    \draw[thin,gray!60] (0,0) grid (2,2);
    \bvert{1}{1} \bhoriz{0}{0} \bjelbow{1}{0}
    \node at (2.5,1) {$\leftrightarrow$};
    \begin{scope}[shift={(3,0)}]
      \draw[thin,gray!60] (0,0) grid (2,2);
      \brelbow{0}{1} \bjelbow{1}{1} \bjelbow{0}{0}
    \end{scope}
  \end{scope}
  \begin{scope}[shift={(21,10)}]
    \draw[thin,gray!60] (0,0) grid (2,2);
    \bvert{1}{1} \brelbow{0}{0} \bjelbow{1}{0}
    \node at (2.5,1) {$\leftrightarrow$};
    \begin{scope}[shift={(3,0)}]
      \draw[thin,gray!60] (0,0) grid (2,2);
      \brelbow{0}{1} \bjelbow{1}{1} \bvert{0}{0}
    \end{scope}
  \end{scope}
  \begin{scope}[shift={(0,6.5)}]
    \draw[thin,gray!60] (0,0) grid (2,2);
    \bjelbow{0}{1} \brelbow{1}{1} \bhoriz{0}{0} \bjelbow{1}{0}
    \node at (2.5,1) {$\leftrightarrow$};
    \begin{scope}[shift={(3,0)}]
      \draw[thin,gray!60] (0,0) grid (2,2);
      \bcross{0}{1} \bhoriz{1}{1} \bjelbow{0}{0}
    \end{scope}
  \end{scope}
  \begin{scope}[shift={(7,6.5)}]
    \draw[thin,gray!60] (0,0) grid (2,2);
    \bjelbow{0}{1} \brelbow{1}{1} \brelbow{0}{0} \bjelbow{1}{0}
    \node at (2.5,1) {$\leftrightarrow$};
    \begin{scope}[shift={(3,0)}]
      \draw[thin,gray!60] (0,0) grid (2,2);
      \bcross{0}{1} \bhoriz{1}{1} \bvert{0}{0}
    \end{scope}
  \end{scope}
  \begin{scope}[shift={(14,6.5)}]
    \draw[thin,gray!60] (0,0) grid (2,2);
    \bjelbow{0}{1} \bvert{1}{1} \bhoriz{0}{0} \bjelbow{1}{0}
    \node at (2.5,1) {$\leftrightarrow$};
    \begin{scope}[shift={(3,0)}]
      \draw[thin,gray!60] (0,0) grid (2,2);
      \bcross{0}{1} \bjelbow{1}{1} \bjelbow{0}{0}
    \end{scope}
  \end{scope}
  \begin{scope}[shift={(21,6.5)}]
    \draw[thin,gray!60] (0,0) grid (2,2);
    \bjelbow{0}{1} \bvert{1}{1} \brelbow{0}{0} \bjelbow{1}{0}
    \node at (2.5,1) {$\leftrightarrow$};
    \begin{scope}[shift={(3,0)}]
      \draw[thin,gray!60] (0,0) grid (2,2);
      \bcross{0}{1} \bjelbow{1}{1} \bvert{0}{0}
    \end{scope}
  \end{scope}
  \begin{scope}[shift={(0,3.5)}]
    \draw[thin,gray!60] (0,0) grid (2,2);
    \brelbow{1}{1} \bhoriz{0}{0} \bcross{1}{0}
    \node at (2.5,1) {$\leftrightarrow$};
    \begin{scope}[shift={(3,0)}]
      \draw[thin,gray!60] (0,0) grid (2,2);
      \brelbow{0}{1} \bhoriz{1}{1} \bjelbow{0}{0} \brelbow{1}{0}
    \end{scope}
  \end{scope}
  \begin{scope}[shift={(7,3.5)}]
    \draw[thin,gray!60] (0,0) grid (2,2);
    \brelbow{1}{1} \brelbow{0}{0} \bcross{1}{0}
    \node at (2.5,1) {$\leftrightarrow$};
    \begin{scope}[shift={(3,0)}]
      \draw[thin,gray!60] (0,0) grid (2,2);
      \brelbow{0}{1} \bhoriz{1}{1} \bvert{0}{0} \brelbow{1}{0}
    \end{scope}
  \end{scope}
  \begin{scope}[shift={(14,3.5)}]
    \draw[thin,gray!60] (0,0) grid (2,2);
    \bvert{1}{1} \bhoriz{0}{0} \bcross{1}{0}
    \node at (2.5,1) {$\leftrightarrow$};
    \begin{scope}[shift={(3,0)}]
      \draw[thin,gray!60] (0,0) grid (2,2);
      \brelbow{0}{1} \bjelbow{1}{1} \bjelbow{0}{0} \brelbow{1}{0}
    \end{scope}
  \end{scope}
  \begin{scope}[shift={(21,3.5)}]
    \draw[thin,gray!60] (0,0) grid (2,2);
    \bvert{1}{1} \brelbow{0}{0} \bcross{1}{0}
    \node at (2.5,1) {$\leftrightarrow$};
    \begin{scope}[shift={(3,0)}]
      \draw[thin,gray!60] (0,0) grid (2,2);
      \brelbow{0}{1} \bjelbow{1}{1} \bvert{0}{0} \brelbow{1}{0}
    \end{scope}
  \end{scope}
  \begin{scope}[shift={(0,0)}]
    \draw[thin,gray!60] (0,0) grid (2,2);
    \bjelbow{0}{1} \brelbow{1}{1} \bhoriz{0}{0} \bcross{1}{0}
    \node at (2.5,1) {$\leftrightarrow$};
    \begin{scope}[shift={(3,0)}]
      \draw[thin,gray!60] (0,0) grid (2,2);
      \bcross{0}{1} \bhoriz{1}{1} \bjelbow{0}{0} \brelbow{1}{0}
    \end{scope}
  \end{scope}
  \begin{scope}[shift={(7,0)}]
    \draw[thin,gray!60] (0,0) grid (2,2);
    \bjelbow{0}{1} \brelbow{1}{1} \brelbow{0}{0} \bcross{1}{0}
    \node at (2.5,1) {$\leftrightarrow$};
    \begin{scope}[shift={(3,0)}]
      \draw[thin,gray!60] (0,0) grid (2,2);
      \bcross{0}{1} \bhoriz{1}{1} \bvert{0}{0} \brelbow{1}{0}
    \end{scope}
  \end{scope}
  \begin{scope}[shift={(14,0)}]
    \draw[thin,gray!60] (0,0) grid (2,2);
    \bjelbow{0}{1} \bvert{1}{1} \bhoriz{0}{0} \bcross{1}{0}
    \node at (2.5,1) {$\leftrightarrow$};
    \begin{scope}[shift={(3,0)}]
      \draw[thin,gray!60] (0,0) grid (2,2);
      \bcross{0}{1} \bjelbow{1}{1} \bjelbow{0}{0} \brelbow{1}{0}
    \end{scope}
  \end{scope}
  \begin{scope}[shift={(21,0)}]
    \draw[thin,gray!60] (0,0) grid (2,2);
    \bjelbow{0}{1} \bvert{1}{1} \brelbow{0}{0} \bcross{1}{0}
    \node at (2.5,1) {$\leftrightarrow$};
    \begin{scope}[shift={(3,0)}]
      \draw[thin,gray!60] (0,0) grid (2,2);
      \bcross{0}{1} \bjelbow{1}{1} \bvert{0}{0} \brelbow{1}{0}
    \end{scope}
  \end{scope}
\end{tikzpicture}
\caption{All flips on $2 \times 2$ windows of BPD tiles.}
\label{fig:local_flips}
\end{figure}
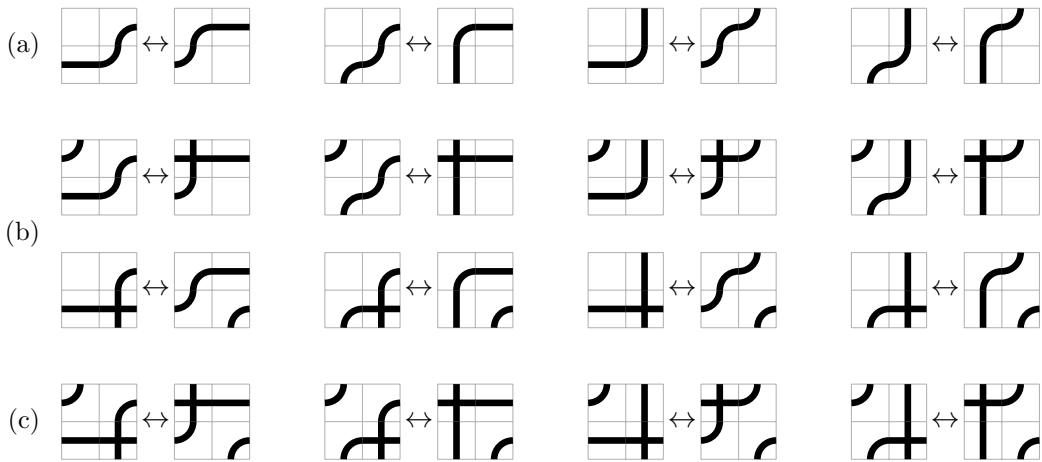

A natural dynamical system on the space of the ASMs
proceeds by \emph{flips}
on $2 \times 2$ windows of tiles:
whenever a $2 \times 2$ block admits an alternative valid tiling
with the same boundary conditions,
the two configurations can be exchanged.
Each flip shifts the height function at the
shared interior vertex by $\pm 1$, while preserving the height values everywhere else.
\Cref{fig:local_flips} displays all possible flips,
organized into three types based on the presence of cross tiles:
\begin{enumerate}[(a)]
	\item {\em Drips}: no cross tile is involved, and the boundary permutation is preserved.
	\item {\em Cross creation/annihilation}: a cross tile appears or disappears, either at the top-left corner (upper sub-row) or at the bottom-right corner (lower sub-row).
	\item {\em Cross relocation}: the cross tile shifts from one corner to the opposite.
 \end{enumerate}
Flips of types~(b) and~(c) alter the global pipe crossing pattern
and may change the boundary permutation~$w$.
In particular, a flip can turn a reduced BPD into a non-reduced one, or vice versa.

These flips define a Markov chain on the ASM lattice:
at each step, choose an interior dual-lattice vertex $(i,j)$
with $1 \le i,j \le n-1$ uniformly at random
and a direction (up or down) with equal probability,
then apply the flip to the four cells sharing vertex $(i,j)$
if one exists
(otherwise the state is unchanged). It is not hard to see that flips connect the state space: any two BPDs (or equivalently, any two ASMs) are related by a sequence of flips. 
This chain is ergodic and reversible with respect to the uniform measure
on all $n \times n$ ASMs.

Restricting to the subset of {\em reduced} BPDs, we conjecture that the same $2\times 2$ flips suffice to connect the state space:

\begin{conjecture}[Connectivity of the RBPD graph]
\label{conj:rbpd_connectivity}
For every $n$, the set of reduced bumpless pipe dreams of size~$n$
is connected under the $2 \times 2$ flips
(\Cref{fig:local_flips})
that preserve reducedness.
\end{conjecture}

We have verified \Cref{conj:rbpd_connectivity} computationally
for $n \le 8$ by
checking that the graph whose vertices are RBPDs and edges are flips that preserve reducedness is connected.
The remainder of this section assumes the conjecture holds.
We discuss obstructions to ``easy'' proofs of this
conjecture in \Cref{sec:traps} below.

\subsection{Failure of monotone CFTP}
\label{sec:cftp_obstruction}

We now explain why the standard Propp--Wilson
monotone Coupling From The Past (CFTP) algorithm \cite{ProppWilsonCP}
cannot be used to sample uniformly from RBPDs of size~$n$.
We begin with a concrete counterexample at $n=4$,
then state the general obstruction.

At $n=4$, the ASM lattice has $42$ elements.
Exactly $41$ are reduced BPDs and one is not.
The non-reduced element $\mathbf{b}^*$ is the unique $4 \times 4$ BPD
in which a single pair of pipes crosses twice
(\Cref{fig:nonreduced_bpd}, right).
Recall that the meet $A \wedge B$ in the ASM lattice
is the BPD whose height function is the pointwise minimum
$h_{A \wedge B}(v) = \min\bigl(h_A(v), h_B(v)\bigr)$.
One can check that among the $\binom{41}{2}$ pairs of RBPDs,
there exist nine pairs $\{A,B\}$ whose
meet $A \wedge B$
equals $\mathbf{b}^*$.
One such pair is shown in \Cref{fig:nonreduced_bpd}, left and center.
Thus, we have the following result:

\begin{proposition}[Failure of the sublattice property]
\label{prop:non_lattice}
For $n \ge 4$,
the set of reduced bumpless pipe dreams
does not form a sublattice of the ASM lattice.
\end{proposition}
\begin{proof}
The $n = 4$ case is verified directly as described above.
For $n \ge 5$, the same pair of RBPDs embeds into the bottom-right
$4 \times 4$ block of the $n \times n$ grid; the remaining
$n-4$ pipes simply travel up and then right without
interacting with the $4 \times 4$ block.
The meet of this embedded pair remains non-reduced.
\end{proof}

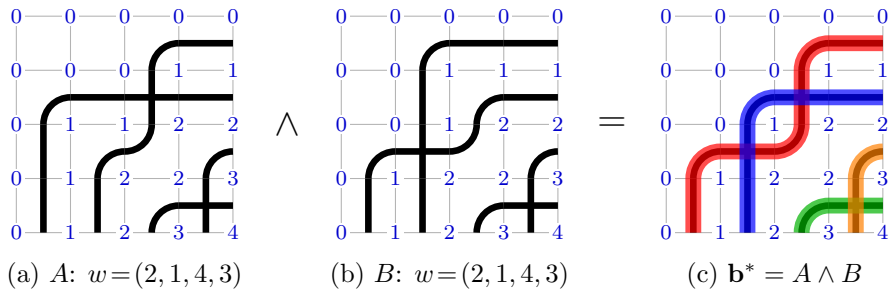
\begin{figure}[htbp]
\centering
\begin{tikzpicture}[scale=0.72]
  \begin{scope}[shift={(0,0)}]
  \draw[thin, gray!60] (0,0) grid (4,4);
  \brelbow{2}{3} \bhoriz{3}{3}
  \brelbow{0}{2} \bhoriz{1}{2} \bcross{2}{2} \bhoriz{3}{2}
  \bvert{0}{1} \brelbow{1}{1} \bjelbow{2}{1} \brelbow{3}{1}
  \bvert{0}{0} \bvert{1}{0} \brelbow{2}{0} \bcross{3}{0}
  \foreach \x/\y/\h in {
    0/4/0, 1/4/0, 2/4/0, 3/4/0, 4/4/0,
    0/3/0, 1/3/0, 2/3/0, 3/3/1, 4/3/1,
    0/2/0, 1/2/1, 2/2/1, 3/2/2, 4/2/2,
    0/1/0, 1/1/1, 2/1/2, 3/1/2, 4/1/3,
    0/0/0, 1/0/1, 2/0/2, 3/0/3, 4/0/4
  } {
    \node[fill=white, inner sep=1pt, font=\scriptsize, text=blue!80!black] at (\x,\y) {\h};
  }
  \node[below, font=\small] at (2,-0.35) {(a) $A$: $w\!=\!(2,1,4,3)$};
  \end{scope}
  \node[font=\Large] at (5.0, 2) {$\wedge$};
  %
  \begin{scope}[shift={(6,0)}]
  \draw[thin, gray!60] (0,0) grid (4,4);
  \brelbow{1}{3} \bhoriz{2}{3} \bhoriz{3}{3}
  \bvert{1}{2} \brelbow{2}{2} \bhoriz{3}{2}
  \brelbow{0}{1} \bcross{1}{1} \bjelbow{2}{1} \brelbow{3}{1}
  \bvert{0}{0} \bvert{1}{0} \brelbow{2}{0} \bcross{3}{0}
  \foreach \x/\y/\h in {
    0/4/0, 1/4/0, 2/4/0, 3/4/0, 4/4/0,
    0/3/0, 1/3/0, 2/3/1, 3/3/1, 4/3/1,
    0/2/0, 1/2/0, 2/2/1, 3/2/2, 4/2/2,
    0/1/0, 1/1/1, 2/1/2, 3/1/2, 4/1/3,
    0/0/0, 1/0/1, 2/0/2, 3/0/3, 4/0/4
  } {
    \node[fill=white, inner sep=1pt, font=\scriptsize, text=blue!80!black] at (\x,\y) {\h};
  }
  \node[below, font=\small] at (2,-0.35) {(b) $B$: $w\!=\!(2,1,4,3)$};
  \end{scope}
  \node[font=\Large] at (11.0, 2) {$=$};
  %
  \begin{scope}[shift={(12,0)}]
  \draw[thin, gray!60] (0,0) grid (4,4);
  \brelbow{2}{3} \bhoriz{3}{3}
  \brelbow{1}{2} \bcross{2}{2} \bhoriz{3}{2}
  \brelbow{0}{1} \bcross{1}{1} \bjelbow{2}{1} \brelbow{3}{1}
  \bvert{0}{0} \bvert{1}{0} \brelbow{2}{0} \bcross{3}{0}
  \draw[red, line width=6pt, opacity=0.7]
    (0.5,0) -- (0.5,1)
    to[out=90,in=180] (1,1.5)
    -- (2,1.5)
    to[out=0,in=270] (2.5,2)
    -- (2.5,3)
    to[out=90,in=180] (3,3.5)
    -- (4,3.5);
  \draw[blue, line width=6pt, opacity=0.7]
    (1.5,0) -- (1.5,2)
    to[out=90,in=180] (2,2.5)
    -- (4,2.5);
  \draw[green!70!black, line width=6pt, opacity=0.7]
    (2.5,0) to[out=90,in=180] (3,0.5)
    -- (4,0.5);
  \draw[orange, line width=6pt, opacity=0.7]
    (3.5,0) -- (3.5,1)
    to[out=90,in=180] (4,1.5);
  \foreach \x/\y/\h in {
    0/4/0, 1/4/0, 2/4/0, 3/4/0, 4/4/0,
    0/3/0, 1/3/0, 2/3/0, 3/3/1, 4/3/1,
    0/2/0, 1/2/0, 2/2/1, 3/2/2, 4/2/2,
    0/1/0, 1/1/1, 2/1/2, 3/1/2, 4/1/3,
    0/0/0, 1/0/1, 2/0/2, 3/0/3, 4/0/4
  } {
    \node[fill=white, inner sep=1pt, font=\scriptsize, text=blue!80!black] at (\x,\y) {\h};
  }
  \node[below, font=\small] at (2,-0.35) {(c) $\mathbf{b}^* = A \wedge B$};
  \end{scope}
\end{tikzpicture}
\caption{Failure of the sublattice property for RBPDs at $n=4$.}
\label{fig:nonreduced_bpd}
\end{figure}

The failure of the sublattice property
has a direct consequence for CFTP.
The Propp--Wilson algorithm \cite{ProppWilsonCP}
requires the random update map $\Phi$ to be \emph{monotone}:
if $X \le Y$ in the partial order, then $\Phi(X) \le \Phi(Y)$ with probability one.
When this holds, it suffices to track the two extremal chains
starting from the unique minimum and maximum elements of the state space,
$\mathbf{b}_{w_0}$ and $\mathbf{b}_{\mathrm{id}}$ (see \Cref{fig:height_function}).
The coalescence
of these two extremal chains (together with the time-doubling procedure of the CFTP algorithm)
guarantees that all intermediate chains
have also coalesced (the ``sandwiching'' property),
yielding an exact uniform sample.

To restrict the flip dynamics
on the ASM lattice
to RBPDs,
one must reject any flip that produces a non-reduced BPD.
One may consider two natural rejection schemes
which still deal with two extremal chains $X$ and $Y$ starting from $\mathbf{b}_{w_0}$ and $\mathbf{b}_{\mathrm{id}}$,
respectively:

\begin{enumerate}[{\bf 1.}]
\item \textbf{Internal (per-chain) rejection.}
Each chain independently checks whether
the proposed flip preserves reducedness,
and rejects it if not.
This scheme preserves the
marginal evolution of $X$ and $Y$ as the original flip dynamics on RBPDs,
but breaks monotonicity:
a shared flip may be accepted by one chain
(which remains reduced) and rejected by the other
(which would become non-reduced),
causing the ordering $X \le Y$ to be violated after the update.
See \Cref{fig:cftp_counterexample_n4} for an example.

\item \textbf{Coupled rejection.}
Reject the flip for
both
chains $X,Y$ whenever
it would make either chain non-reduced.
This preserves monotonicity,
but the transition probabilities of each chain
now depend on the state of the other chain.
The resulting coupling from the past process no longer targets
the uniform distribution on RBPDs.
\end{enumerate}
We see that neither scheme produces a valid monotone CFTP sampler.

\begin{figure}[htbp]
\centering
\begin{tikzpicture}[scale=0.72]
  \begin{scope}[shift={(7,6)}]
    \draw[thin, gray!60] (0,0) grid (4,4);
    \brelbow{2}{3} \bhoriz{3}{3}
    \brelbow{0}{2} \bhoriz{1}{2} \bcross{2}{2} \bhoriz{3}{2}
    \bvert{0}{1} \brelbow{1}{1} \bcross{2}{1} \bhoriz{3}{1}
    \bvert{0}{0} \bvert{1}{0} \bvert{2}{0} \brelbow{3}{0}
    \draw[orange!90!black, dashed, line width=1.2pt] (2,0) rectangle (4,2);
    \fill[blue!70!black] (3,1) circle (3pt);
    \node[font=\scriptsize, blue!70!black] at (4.7,1) {$h\!=\!3$};
    \node[below, font=\small] at (2,-0.35) {(b) $Y$: $w_Y\!=\!(3,1,2,4)$};
  \end{scope}

  \begin{scope}[shift={(0,6)}]
    \draw[thin, gray!60] (0,0) grid (4,4);
    \brelbow{2}{3} \bhoriz{3}{3}
    \brelbow{1}{2} \bcross{2}{2} \bhoriz{3}{2}
    \brelbow{0}{1} \bcross{1}{1} \bcross{2}{1} \bhoriz{3}{1}
    \bvert{0}{0} \bvert{1}{0} \bvert{2}{0} \brelbow{3}{0}
    \draw[orange!90!black, dashed, line width=1.2pt] (2,0) rectangle (4,2);
    \draw[red!70!black, line width=1.2pt, rounded corners=2pt] (1.02,1.02) rectangle (1.98,1.98);
    \fill[blue!70!black] (3,1) circle (3pt);
    \node[font=\scriptsize, blue!70!black] at (4.7,1) {$h\!=\!3$};
    \node[below, font=\small] at (2,-0.35) {(a) $X$: $w_X\!=\!(3,2,1,4)$};
  \end{scope}

  \begin{scope}[shift={(7,0)}]
    \draw[thin, gray!60] (0,0) grid (4,4);
    \brelbow{2}{3} \bhoriz{3}{3}
    \brelbow{0}{2} \bhoriz{1}{2} \bcross{2}{2} \bhoriz{3}{2}
    \bvert{0}{1} \brelbow{1}{1} \bjelbow{2}{1} \brelbow{3}{1}
    \bvert{0}{0} \bvert{1}{0} \brelbow{2}{0} \bcross{3}{0}
    \fill[blue!70!black] (3,1) circle (3pt);
    \node[font=\scriptsize, blue!70!black] at (4.7,1) {$h\!=\!2$};
    \node[below, font=\small] at (2,-0.35) {(d) $Y'$: accepted};
  \end{scope}

  \begin{scope}[shift={(0,0)}]
    \draw[thin, gray!60] (0,0) grid (4,4);
    \brelbow{2}{3} \bhoriz{3}{3}
    \brelbow{1}{2} \bcross{2}{2} \bhoriz{3}{2}
    \brelbow{0}{1} \bcross{1}{1} \bcross{2}{1} \bhoriz{3}{1}
    \bvert{0}{0} \bvert{1}{0} \bvert{2}{0} \brelbow{3}{0}
    \fill[blue!70!black] (3,1) circle (3pt);
    \node[font=\scriptsize, blue!70!black] at (4.7,1) {$h\!=\!3$};
    \node[below, font=\small] at (2,-0.35) {(c) $X'\!=\!X$: rejected};
  \end{scope}
\end{tikzpicture}
\caption{Monotonicity violation at $n=4$ under internal rejection.
States $X \le Y$ have identical $2\times2$ windows around $v=(3,3)$
(dashed box).
The flip is accepted by~$Y$
but rejected by~$X$ due to the extra cross outside the window (highlighted in $X$).
This yields $X' \nleq Y'$ since
$h_{X'}(3,3) = 3$ but  $h_{Y'}(3,3)= 2$.}
\label{fig:cftp_counterexample_n4}
\end{figure}
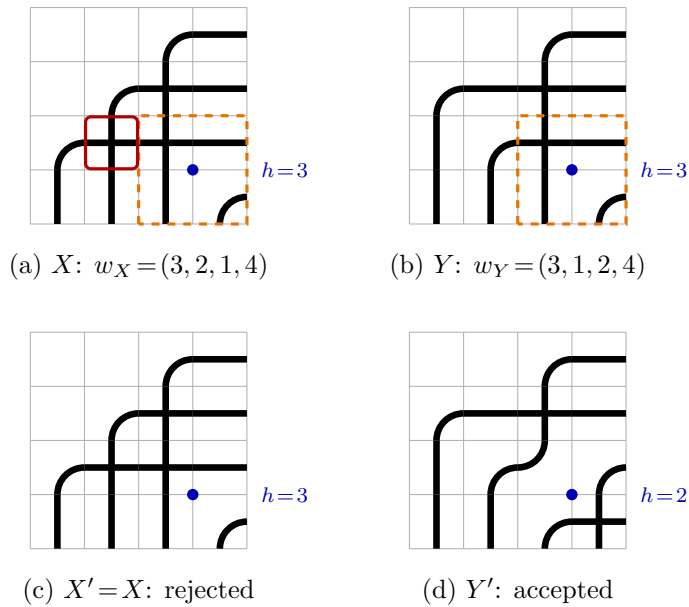

Because reducedness depends on the global pipe crossing pattern,
the height function at the flip site cannot detect whether
a distant cross will cause a double crossing.
The direct enumeration, presented in \Cref{tab:monotonicity},
shows that these monotonicity violations
appear as soon as $n = 4$ and grow with~$n$.

\begin{table}[htbp]
    \centering
    \caption{Monotonicity violations
		under internal rejection.
		For each ordered pair $X \le Y$ of RBPDs,
		all $2(n-1)^2$ flips are tested
		($(n-1)^2$ interior vertices $\times$ $2$ directions).
		A violation occurs when a shared flip is accepted by one chain
		but rejected by the other, breaking the order $X' \le Y'$.}
    \label{tab:monotonicity}
    \begin{tabular}{crrrc}
        \toprule
        $n$ & $|$RBPDs$|$ & Ordered pairs & Flip checks & Violations \\
        \midrule
        3 & 7 & 26 & 208 & \textbf{0} \\
        4 & 41 & 618 & 11{,}124 & \textbf{16} \\
        5 & 393 & 39{,}302 & 1{,}257{,}664 & \textbf{2{,}259} \\
        \bottomrule
    \end{tabular}
\end{table}

\subsection{False coalescence}
\label{sec:false_coalescence}

The internal rejection mechanism preserves the correct
stationary distribution but breaks monotonicity.
Here we directly test whether the ``naive''
version of the CFTP algorithm that tracks only the extremal chains
$X$ and $Y$
(started from $\mathbf{b}_{w_0}$ and $\mathbf{b}_{\mathrm{id}}$)
can still be used to sample from the uniform distribution on RBPDs.
Indeed, one might hope that the extremal chains
rarely cross, and their coalescence
may be close to the universal coalescence.

For each trial, we run the standard backward CFTP protocol
with time doubling \cite{ProppWilsonCP}:
for $T=1,2,4,8,\ldots$, we draw fresh random updates
for times $-2T,\ldots,-(T+1)$, place them before the
existing updates for $-T,\ldots,-1$, and apply the combined
sequence of Markov steps in the two extremal chains,
doubling $T$ until the extremal chains coalesce at time $0$.
We record the final update sequence.
We then replay \emph{the same sequence}
starting from \emph{every} RBPD in the state space,
and check whether all chains arrive at the same final state.
We say that \emph{false coalescence} occurs when
$X(0) = Y(0)$, but for some intermediate chain $Z$ started from a different RBPD, $Z(0) \ne X(0)$.
The simulation results (\Cref{tab:false_coalescence}) confirm that
this is not a rare event: at $n=5$,
over $18\%$ of CFTP terminations are false.

\begin{table}[htbp]
    \centering
    \caption{False coalescence rates.
		Each trial runs backward CFTP until the extremal chains coalesce,
		then replays from all starting states.
		A failure means the extremal chains agreed
		but at least one intermediate chain did not.}
    \label{tab:false_coalescence}
    \begin{tabular}{crrcc}
        \toprule
        $n$ & $|\mathrm{RBPDs}|$ & Trials & False coalescences & Rate \\
        \midrule
        3 & 7 & 1{,}000 & 0 & 0.0\% \\
        4 & 41 & 500{,}000 & 37{,}476 & \textbf{7.5\%} \\
        5 & 393 & 50{,}000 & 9{,}361 & \textbf{18.7\%} \\
        \bottomrule
    \end{tabular}
\end{table}

To confirm that this sampling algorithm produces a biased output distribution,
we compare the naive CFTP output at $n=4$
to the target distribution on permutations.
A uniform distribution on RBPDs induces a distribution
on permutations $w \in S_n$ with probability proportional
to $\Upsilon_w = \mathfrak{S}_w(1^n)$
(the number of reduced BPDs for~$w$).
Over $500{,}000$ naive CFTP samples,
the permutation frequencies
deviate significantly from this target:
a Pearson $\chi^2$ test yields $\chi^2 = 60.7$
with $23$ degrees of freedom ($p \approx 3 \times 10^{-5}$).
\Cref{tab:cftp_bias} shows the per-permutation breakdown,
grouped by $\Upsilon_w$.
The per-permutation relative deviations are around $2\%$,
but the bias is statistically significant.

\begin{table}[htbp]
    \centering
    \caption{Naive CFTP output distribution on permutations at $n=4$
		($500{,}000$ trials).
        Permutations are grouped by $\Upsilon_w$.
        Under an unbiased sampler, the ratios of observed/expected probability
				should be much closer to $1.00$.
				Expected counts are rounded from $500{,}000 \cdot \Upsilon_w / 41$.}
    \label{tab:cftp_bias}
    \begin{tabular}{clrrrr}
        \toprule
        $\Upsilon_w$ & $w$ & Expected & Observed & Obs/Exp & $(\mathrm{O{-}E})^2/\mathrm{E}$ \\
        \midrule
        5 & $(1,4,3,2)$ & 60{,}976 & 61{,}607 & 1.010 & 6.5 \\
        \midrule
        3 & $(1,2,4,3)$ & 36{,}585 & 36{,}093 & 0.987 & 6.6 \\
          & $(1,3,4,2)$ & 36{,}585 & 36{,}669 & 1.002 & 0.2 \\
          & $(1,4,2,3)$ & 36{,}585 & 36{,}848 & 1.007 & 1.9 \\
          & $(2,1,4,3)$ & 36{,}585 & 36{,}545 & 0.999 & 0.0 \\
        \midrule
        2 & $(1,3,2,4)$ & 24{,}390 & 23{,}924 & 0.981 & 8.9 \\
          & $(2,4,1,3)$ & 24{,}390 & 24{,}489 & 1.004 & 0.4 \\
          & $(2,4,3,1)$ & 24{,}390 & 24{,}536 & 1.006 & 0.9 \\
          & $(3,1,4,2)$ & 24{,}390 & 24{,}669 & 1.011 & 3.2 \\
          & $(4,1,3,2)$ & 24{,}390 & 24{,}546 & 1.006 & 1.0 \\
        \midrule
        1 & \multicolumn{5}{l}{14 permutations (expected 12{,}195 each):} \\
          & max obs/exp & 12{,}195 & 12{,}447 & 1.021 & 5.2 \\
          & min obs/exp & 12{,}195 & 11{,}905 & 0.976 & 6.9 \\
        \bottomrule
    \end{tabular}
\end{table}

We therefore abandon exact sampling via CFTP and instead,
in \Cref{sec:mcmc_sampler} below,
develop a Markov chain Monte Carlo (MCMC)
sampler for uniformly random RBPDs.

\section{Sampling: MCMC for reduced bumpless pipe dreams}
\label{sec:mcmc_sampler}

We turn to a Markov Chain Monte Carlo (MCMC) approach
for sampling uniformly random RBPDs:
a random walk on the state space of RBPDs whose stationary distribution
is the uniform measure.
In \Cref{sec:traps}, we further discuss
\Cref{conj:rbpd_connectivity}
(connectivity of the state space under the 
$2 \times 2$ flips that preserve reducedness),
and present examples of ``traps'' that require non-monotone paths to escape.
In \Cref{sec:droops_connectivity}, we introduce droops and undroops ---
rectangular moves from \cite{LamLeeShimozono2021BPD} that
bypass the traps --- and prove that
the combined move set makes the state space connected,
establishing the correctness of the MCMC sampler.

\subsection{Traps for local flips}
\label{sec:traps}

We verified \Cref{conj:rbpd_connectivity}
for $n \le 8$ by direct exploration of the state space.
This exploration also revealed the presence of ``traps'' in the state space
that require non-monotone paths to escape.

\begin{definition}
\label{def:down_up_flips}
	Recall the $2\times 2$ flips in \Cref{fig:local_flips}. Call a flip \emph{up} if
	it \emph{increases} the height function at the flip site by one
	(bringing the RBPD closer to the maximum $\mathbf{b}_{\mathrm{id}}$),
	and \emph{down}
	otherwise (bringing it closer to the minimum $\mathbf{b}_{w_0}$).
	In \Cref{fig:height_function}, up flips are replacing the left local configuration with the right one,
	and down flips are the reverse.
\end{definition}

One might hope that any RBPD can be transformed into $\mathbf{b}_{\mathrm{id}}$ by a sequence of up flips,
but this is not the case. We call a
RBPD $\mathbf{b}\ne \mathbf{b}_{\mathrm{id}}$ \emph{stuck} if it
does not admit any up flips (that preserve reducedness).

For $n=8$, we found exactly ten such stuck RBPDs,
all exhibiting the same mechanism, the \emph{mutual frustration} of two j-elbows
(\Cref{fig:mutual_frustration}, left).
Each of the two j-elbows is locked by a triple crossing
which prevents it from being flipped up. Moreover,
every box-cross annihilation near one of the locked
j-elbows would create a double crossing, which is forbidden.
The state space remains connected
via non-monotone paths:
escaping the trap requires down flips.
Another example, with three j-elbows at $n=9$, is shown in
\Cref{fig:mutual_frustration}, right.

\begin{figure}[htbp]
\centering
\begin{tikzpicture}[scale=0.675]
  \draw[thin, gray] (0,0) grid (8,8);
  \bbox{0}{7} \bbox{1}{7} \bbox{2}{7} \bbox{3}{7}
  \brelbow{4}{7} \bhoriz{5}{7} \bhoriz{6}{7} \bhoriz{7}{7}
  \bbox{0}{6} \bbox{1}{6} \bbox{2}{6} \bbox{3}{6}
  \bvert{4}{6} \brelbow{5}{6} \bhoriz{6}{6} \bhoriz{7}{6}
  \brelbow{0}{5} \bhoriz{1}{5} \bhoriz{2}{5} \bhoriz{3}{5}
  \bcross{4}{5} \bcross{5}{5} \bhoriz{6}{5} \bhoriz{7}{5}
  \bvert{0}{4} \bbox{1}{4} \bbox{2}{4} \brelbow{3}{4}
  \bcross{4}{4} \bjelbow{5}{4} \brelbow{6}{4} \bhoriz{7}{4}
  \bvert{0}{3} \bbox{1}{3} \bbox{2}{3} \bvert{3}{3}
  \bvert{4}{3} \brelbow{5}{3} \bcross{6}{3} \bhoriz{7}{3}
  \bvert{0}{2} \brelbow{1}{2} \bhoriz{2}{2}
  \bcross{3}{2} \bcross{4}{2} \bcross{5}{2} \bcross{6}{2} \bhoriz{7}{2}
  \bvert{0}{1} \bvert{1}{1} \brelbow{2}{1}
  \bcross{3}{1} \bcross{4}{1} \bjelbow{5}{1} \bvert{6}{1} \brelbow{7}{1}
  \bvert{0}{0} \bvert{1}{0} \bvert{2}{0} \bvert{3}{0}
  \bvert{4}{0} \brelbow{5}{0} \bcross{6}{0} \bcross{7}{0}
  \draw[red, line width=4pt, rounded corners=3pt]
    (5,4) rectangle (6,5);
  \draw[red, line width=4pt, rounded corners=3pt]
    (5,1) rectangle (6,2);
\end{tikzpicture}
\hspace{2cm}
\begin{tikzpicture}[scale=0.6]
  \draw[thin, gray] (0,0) grid (9,9);
  \bbox{0}{8} \bbox{1}{8} \bbox{2}{8} \bbox{3}{8} \bbox{4}{8}
  \brelbow{5}{8} \bhoriz{6}{8} \bhoriz{7}{8} \bhoriz{8}{8}
  \bbox{0}{7} \bbox{1}{7} \bbox{2}{7} \bbox{3}{7} \bbox{4}{7}
  \bvert{5}{7} \brelbow{6}{7} \bhoriz{7}{7} \bhoriz{8}{7}
  \bbox{0}{6} \bbox{1}{6} \brelbow{2}{6} \bhoriz{3}{6} \bhoriz{4}{6}
  \bcross{5}{6} \bcross{6}{6} \bhoriz{7}{6} \bhoriz{8}{6}
  \bbox{0}{5} \bbox{1}{5} \bvert{2}{5} \bbox{3}{5} \brelbow{4}{5}
  \bcross{5}{5} \bjelbow{6}{5} \brelbow{7}{5} \bhoriz{8}{5}
  \bbox{0}{4} \bbox{1}{4} \bvert{2}{4} \brelbow{3}{4} \bcross{4}{4}
  \bcross{5}{4} \bhoriz{6}{4} \bcross{7}{4} \bhoriz{8}{4}
  \bbox{0}{3} \bbox{1}{3} \bvert{2}{3} \bvert{3}{3} \bvert{4}{3}
  \bvert{5}{3} \brelbow{6}{3} \bjelbow{7}{3} \brelbow{8}{3}
  \brelbow{0}{2} \bhoriz{1}{2} \bcross{2}{2} \bcross{3}{2} \bcross{4}{2}
  \bcross{5}{2} \bcross{6}{2} \bhoriz{7}{2} \bcross{8}{2}
  \bvert{0}{1} \brelbow{1}{1} \bcross{2}{1} \bjelbow{3}{1} \bvert{4}{1}
  \bvert{5}{1} \bvert{6}{1} \brelbow{7}{1} \bcross{8}{1}
  \bvert{0}{0} \bvert{1}{0} \bvert{2}{0} \brelbow{3}{0} \bcross{4}{0}
  \bcross{5}{0} \bcross{6}{0} \bcross{7}{0} \bcross{8}{0}
  \draw[red, line width=4pt, rounded corners=3pt]
    (6,5) rectangle (7,6);
  \draw[red, line width=4pt, rounded corners=3pt]
    (7,3) rectangle (8,4);
  \draw[red, line width=4pt, rounded corners=3pt]
    (3,1) rectangle (4,2);
\end{tikzpicture}
\caption{Examples of stuck RBPDs for $n=8$ (left) and $n=9$ (right).}
\label{fig:mutual_frustration}
\end{figure}
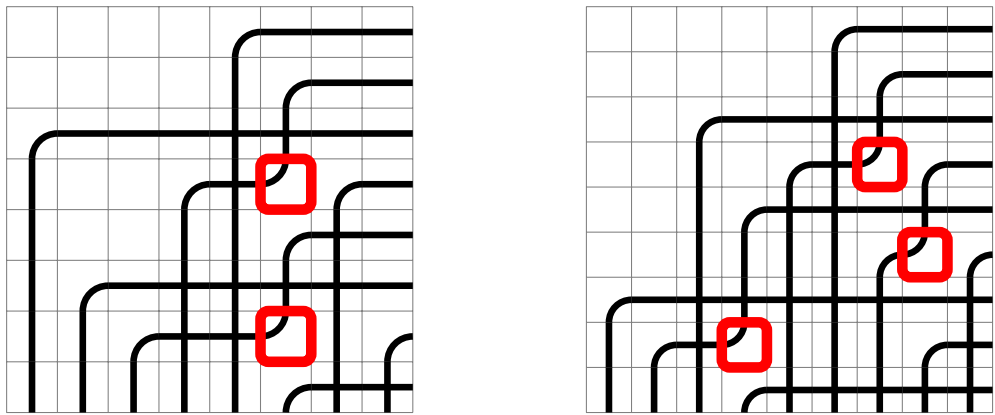

\subsection{Droops, undroops, and connectivity}
\label{sec:droops_connectivity}

The traps of \Cref{sec:traps} show that
establishing connectivity of the RBPD graph under
the $2\times 2$ flips alone (\Cref{conj:rbpd_connectivity}) appears difficult.
While we still believe this conjecture is true, we establish a
weaker connectivity result by supplementing the $2\times 2$ flips with
\emph{droops} and \emph{undroops}, rectangular moves introduced in
\cite{LamLeeShimozono2021BPD}.
We show that the Markov chain with this larger move set is ergodic and
preserves the uniform distribution on the set of all RBPDs, thus providing a
valid MCMC sampler.

\begin{definition}[Droop and undroop {\cite{LamLeeShimozono2021BPD}}]
\label{def:droop}
Let $R = [i_1, i_2] \times [j_1, j_2]$ be a rectangle in an $n \times n$ BPD
with $i_1 < i_2$ and $j_1 < j_2$.
Call $R$ \emph{droopable} if
the NW corner $(i_1, j_2)$ is an r-elbow,
the SE corner $(i_2, j_1)$ is empty,
and $R$ has no elbows except possibly at the four corners
(i.e., every non-corner tile of $R$ is empty, cross, vertical, or horizontal).

The \emph{droop} at $R$ replaces tiles on the boundary of $R$ as follows:
the NW corner becomes empty, the SE corner becomes a j-elbow,
and the NE and SW corners both become r-elbows.
On the four borders (excluding corners):
the north and west borders \emph{retract}
(horizontal $\to$ empty, cross $\to$ vertical on north;
vertical $\to$ empty, cross $\to$ horizontal on west),
while the south and east borders \emph{extend}
(empty $\to$ horizontal, vertical $\to$ cross on south;
empty $\to$ vertical, horizontal $\to$ cross on east).
Interior tiles are unchanged.

The \emph{undroop} is the inverse:
if the NW corner is empty, the SE corner is a j-elbow,
the NE and SW corners are r-elbows,
and there are no interior elbows,
then the undroop restores the r-elbow at NW and the empty tile at SE.
A $3 \times 3$ droop is shown in \Cref{fig:droop_examples}.
\end{definition}

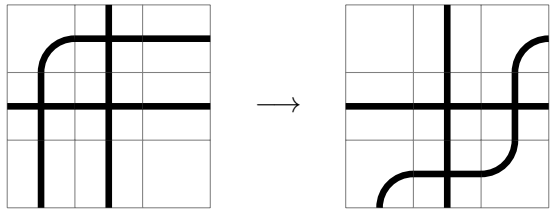
\begin{figure}[htbp]
\centering
\begin{tikzpicture}[scale=0.9]
  \begin{scope}
    \draw[thin, gray] (0,0) grid (3,3);
    \brelbow{0}{2} \bcross{1}{2} \bhoriz{2}{2}
    \bcross{0}{1} \bcross{1}{1} \bhoriz{2}{1}
    \bvert{0}{0} \bvert{1}{0} \bbox{2}{0}
  \end{scope}
  \node at (4,1.5) {$\longrightarrow$};
  \begin{scope}[shift={(5,0)}]
    \draw[thin, gray] (0,0) grid (3,3);
    \bbox{0}{2} \bvert{1}{2} \brelbow{2}{2}
    \bhoriz{0}{1} \bcross{1}{1} \bcross{2}{1}
    \brelbow{0}{0} \bcross{1}{0} \bjelbow{2}{0}
  \end{scope}
\end{tikzpicture}
\caption{A $3\times 3$ droop.}
\label{fig:droop_examples}
\end{figure}

Droops and undroops 
preserve reducedness and the boundary permutation $w$ of an RBPD.
Moreover, by 
\cite[Proposition~5.3]{LamLeeShimozono2021BPD},
for any $w \in S_n$, every $w$-RBPD can be obtained from the
Rothe RBPD $\mathbf{b}_w$ (see \Cref{sec:notation} for the definition)
by a sequence of droops.

To connect RBPDs with \emph{different} boundary permutations,
we use the $2\times 2$ flips:

\begin{lemma}
\label{lem:rothe_admits_flip}
For $w \ne \mathrm{id}$, the Rothe RBPD $\mathbf{b}_w$ admits a
$2 \times 2$ flip of type~(b) (\Cref{fig:local_flips})
that produces an RBPD with boundary permutation $w'$
satisfying $\ell(w') = \ell(w) - 1$.
\end{lemma}
\begin{proof}
The Rothe RBPD $\mathbf{b}_w$ has no j-elbows:
its empty tiles are exactly the cells of the Rothe diagram $D(w)$,
and all other non-cross tiles are r-elbows, vertical, or horizontal.
Since $w \ne \mathrm{id}$, the diagram $D(w)$ is nonempty.  Choose a cross in position $(i,j)$, in the leftmost column where crosses occur (so $j$ is minimal), and such that the cell $(i-1,j-1)$ is empty.  (For instance, the topmost cross in this leftmost column works.) The vertical pipe of this cross comes from column $b=j$ on the bottom (south) edge and exits at row $w(b)$ on the right (east) edge.  The horizontal pipe of the cross comes from column $a=w^{-1}(i)$ on the bottom (south) edge and exits at row $i=w(a)$ on the right (east).  Our labels chosen such that $a<b$ and (since the pipes cross) $w(a)>w(b)$.  From our choice of cross, and the fact that there are no j-elbows, for any $c$ with $a<c<b$, the pipe beginning at column $c$ on the south edge must exit at row $w(c)>w(a)>w(b)$.  It follows that $\ell( w\cdot(a,b) ) = \ell(w)-1$.  See Figure~\ref{fig:rothe-cross-proof}.

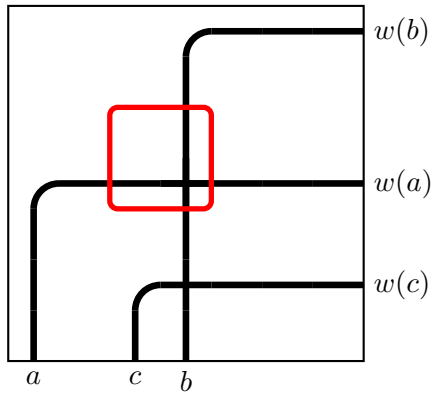
\begin{figure}
\begin{center}
\begin{tikzpicture}[scale=0.675]

  \draw[thick] (0,0) rectangle (7,7);

  \bvert{3}{5} \bvert{3}{4} \bvert{3}{3} \bvert{3}{2} \bvert{3}{1} \bvert{3}{0}
  \bhoriz{4}{6} \bhoriz{5}{6} \bhoriz{6}{6}
  \brelbow{3}{6}

  \bvert{0}{2} \bvert{0}{1} \bvert{0}{0}
  \bhoriz{1}{3} \bhoriz{2}{3} \bhoriz{3}{3} \bhoriz{4}{3} \bhoriz{5}{3} \bhoriz{6}{3}
  \brelbow{0}{3}

   \bvert{2}{0}
  \bhoriz{3}{1} \bhoriz{4}{1} \bhoriz{5}{1} \bhoriz{6}{1}
  \brelbow{2}{1}
  \bcross{3}{3}

  \node[below] at (0.5,0) {$a$};
  \node[below] at (3.5,0) {$b$};
  \node[below] at (2.5,0) {$c$};

  \node[right] at (7,3.5) {$w(a)$};
  \node[right] at (7,6.5) {$w(b)$};
  \node[right] at (7,1.5) {$w(c)$};

    \draw[red, line width=2pt, rounded corners=3pt]
    (2,3) rectangle (4,5);
\end{tikzpicture}
\end{center}
\caption{Schematic view of a cross in the Rothe BPD}
\label{fig:rothe-cross-proof}
\end{figure}

The $2\times2$ configuration bounding around our chosen cross at $(i,j)$ and empty tile at $(i-1,j-1)$  looks like one of the following configurations:
\begin{equation*}
	\begin{tikzpicture}[scale=0.5]
		  \begin{scope}[shift={(0,3.5)}]
    \draw[thin,gray!60] (0,0) grid (2,2);
    \brelbow{1}{1} \bhoriz{0}{0} \bcross{1}{0}
  \end{scope}
  \begin{scope}[shift={(5,3.5)}]
    \draw[thin,gray!60] (0,0) grid (2,2);
    \brelbow{1}{1} \brelbow{0}{0} \bcross{1}{0}
  \end{scope}
  \begin{scope}[shift={(10,3.5)}]
    \draw[thin,gray!60] (0,0) grid (2,2);
    \bvert{1}{1} \bhoriz{0}{0} \bcross{1}{0}
  \end{scope}
  \begin{scope}[shift={(15,3.5)}]
    \draw[thin,gray!60] (0,0) grid (2,2);
    \bvert{1}{1} \brelbow{0}{0} \bcross{1}{0}
  \end{scope}
	\end{tikzpicture}
\end{equation*}
(Highlighted in red in Figure~\ref{fig:rothe-cross-proof}.)  These are exactly the configurations before flips 
in \Cref{fig:local_flips}, namely, in the lower row in type (b).
The corresponding type (b) flip annihilates the chosen cross at $(i, j)$, and the resulting permutation is $w\cdot (a,b)$, which is of length one less than $w$, as observed above.
\end{proof}

\begin{theorem}[Connectivity]
\label{thm:rbpd_connectivity}
	For every $n$, the set of all $n \times n$ RBPDs
	is connected under $2 \times 2$ flips (\Cref{fig:local_flips})
	that preserve reducedness, together with droops and undroops.
\end{theorem}
\begin{proof}
By \cite[Proposition~5.3]{LamLeeShimozono2021BPD}, any RBPD with boundary
permutation $w$ can be undrooped to the Rothe RBPD $\mathbf{b}_w$, and by
\Cref{lem:rothe_admits_flip}, a type-(b) flip then produces an RBPD whose
boundary permutation has length $\ell(w)-1$.
This new RBPD can again be undrooped to its Rothe RBPD, 
and the length can be reduced again, until reaching 
the maximal RBPD $\mathbf{b}_{\mathrm{id}}$.
Since all moves are reversible, any two RBPDs can be connected
by a sequence of $2\times 2$ flips and droops/undroops.
\end{proof}

\Cref{thm:rbpd_connectivity} is weaker than
\Cref{conj:rbpd_connectivity}, which asserts connectivity
under $2\times 2$ flips alone.

\begin{corollary}
\label{cor:mcmc_ergodic}
Consider a random walk on $n \times n$ RBPDs that at each step
attempts one of the $2(n-1)^2$ possible $2\times 2$ flips
(rejecting those that break reducedness)
or one of the $\binom{n}{2}^2$ possible droop/undroop rectangles
(the count corresponds to choosing row and column coordinates for the rectangle),
such that a move and its inverse are attempted with equal probability (but this probability may depend on the move).
Then the walk is ergodic on the space
of all RBPDs, and its stationary distribution
is uniform.
\end{corollary}
\begin{proof}
Connectivity follows from \Cref{thm:rbpd_connectivity},
and aperiodicity from the positive rejection probability.
For detailed balance with respect to the uniform distribution,
note that each $2\times 2$ flip is self-inverse,
and each droop at a rectangle $R$ is paired with the undroop at the same~$R$,
so the symmetry condition ensures equal transition rates
in both directions.
\end{proof}

\begin{remark}
As discussed in \Cref{sec:cftp_failure},
already the $2\times 2$ flips alone preclude a valid CFTP scheme.
Adding droops and undroops to the move set
only compounds the difficulty.
Thus, we do not pursue exact sampling via CFTP,
and instead rely on the approximate MCMC sampler based
on \Cref{cor:mcmc_ergodic}.
\end{remark}

\subsection{MCMC parameters and simulation results}
\label{sec:mcmc_results}

The move set of the MCMC sampler 
developed in \Cref{sec:droops_connectivity} above
consists of:
\begin{enumerate}[$\bullet$]
\item \textbf{Local $2\times 2$ flips}
  (\Cref{sec:local_flips}):
  there are $(n-1)^2$ possible $2\times 2$ windows in an $n\times n$ grid,
  each with two directions (flip up or down),
  giving $2(n-1)^2$ flip moves in total.
\item \textbf{Droops and undroops} (\Cref{def:droop}):
  there are $\binom{n}{2}^2$ candidate rectangles in total.
  Depending on the tile types at the corners of a rectangle,
  each rectangle can lead to either a droop, an undroop, or a non-move 
	(when the rectangle 
	is not droopable or undroopable).
\end{enumerate}
To run the chain, we must specify a \emph{proposal rule}
that determines how the chain selects a candidate move at each step.
There is considerable freedom in this choice
(probabilities of each move type, distribution over rectangle sizes, etc.);
by \Cref{cor:mcmc_ergodic}, any symmetric rule with full support
preserves the uniform stationary distribution,
and different rules affect only the mixing rate.
We use the following:
\begin{definition}[MCMC proposal rule]
	\label{def:mcmc_proposal_rule}
	At each step, with probability $\frac{3}{4}$
	we attempt a $2\times 2$ flip
	(choosing a uniformly random $2\times 2$ window and a random direction,
	up or down, as in \Cref{sec:local_flips}).
	This rule is symmetric in the sense that up and down flips are attempted with equal probability.

	With the complementary probability $\frac{1}{4}$ we attempt a droop or undroop.
	We pick a rectangle $R$ from the $\binom{n}{2}^2$ candidates
	(the specific distribution over rectangle sizes is discussed below)
	and check the corner tiles:
	if the northwest corner is an r-elbow and the southeast corner is empty,
	we perform a droop at $R$;
	if the northwest corner is empty and the southeast corner is a j-elbow,
	we perform an undroop;
	otherwise, the move is rejected.
	The rule is symmetric: each rectangle is selected with
	the same probability regardless of whether it leads to a droop or undroop.
\end{definition}

Reducedness of the BPD is maintained at each step
by tracking the pipe colors throughout the $n\times n$ grid
(see \Cref{fig:bpd_examples_intro}, bottom, for an example)
and incrementally updating the boundary permutation and its inversion count.
A BPD is reduced if and only if the number of cross tiles
equals the number of inversions of the boundary permutation.
For a $2\times 2$ flip,
the changed tiles can redirect pipes passing through the $2\times 2$ block;
we trace each affected pipe to the grid boundary in $O(n)$ time,
updating edge colors along the way.
The flip is accepted if reducedness is preserved.
Droops and undroops automatically preserve reducedness
(they do not change the boundary permutation or the cross count),
so no rejection is needed.
After an accepted droop/undroop, we update edge colors
incrementally within the affected rectangle, and propagate changes
outward, in time $O(k^2 + kn)$, where $k$ is longest side length of the rectangle.

\medskip

Let us now discuss the 
distribution over rectangle sizes for droop proposals.
We pick the southeast corner $(i,j)$ 
of the rectangle uniformly at random from $\{2, \ldots, n\}^2$,
and then pick the offsets $\delta_i\in \{1, \ldots, i-1\}$, $\delta_j \in \{1, \ldots, j-1\}$
to determine the northwest corner $(i-\delta_i, j-\delta_j)$ of the rectangle
(it must be inside the $n\times n$ grid).
For exploratory diagnostics, we tested the following distributions for the offsets
(here, $\delta_{\max}$ stands for $i-1$ or $j-1$, depending on the coordinate):
\begin{enumerate}[{\bf 1.}]
\item
\emph{Geometric}: $\Prob(\delta = k) \propto 2^{-k}$, $1\le k \le \delta_{\max}$.
\item
\emph{Uniform}: $\Prob(\delta = k)$ constant on $\{1, \ldots, \delta_{\max}\}$.
\item
\emph{Log-uniform} (\emph{reciprocal}): $\Prob(\delta = k) \propto 1/k$, $1\le k \le \delta_{\max}$.
\item
\emph{Reverse log-uniform} (\emph{reverse reciprocal}):
$\Prob(\delta = k) \propto 1/(\delta_{\max}+1-k)$, $1\le k \le \delta_{\max}$.
\end{enumerate}

\begin{figure}[htbp]
\centering
\includegraphics[width=\textwidth]{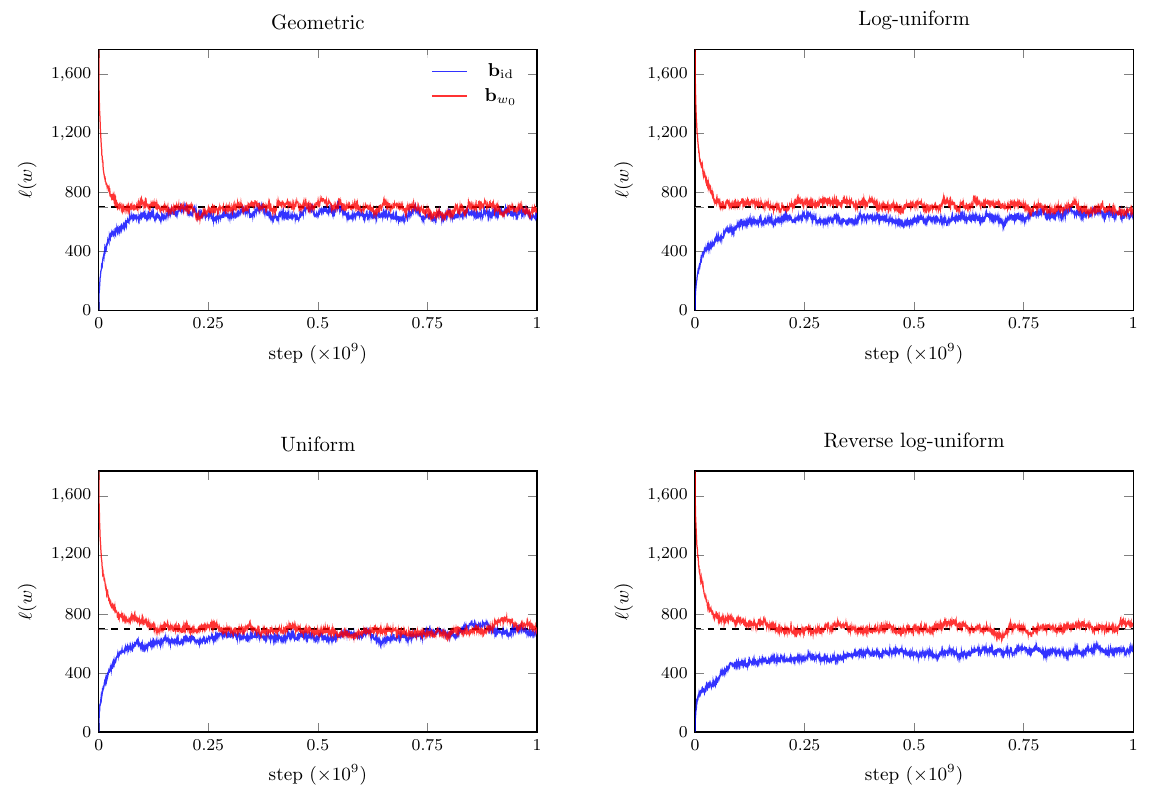}
\caption{Plots of $\ell(w)$ (number of crosses in a BPD, equivalently, length of the boundary permutation)
for the four rectangle size distributions at $n=60$,
starting from $\mathbf{b}_{\mathrm{id}}$ (blue; starting from $\ell=0$)
and $\mathbf{b}_{w_0}$ (red; starting from $\ell=\binom n2=1770$).}
\label{fig:mixing_comparison}
\end{figure}

The diagnostic experiments are performed for $n = 60$.
Each experiment consists of initial \emph{burn-in} steps, followed by
\emph{collection} steps, where we record samples at regular intervals
(thinning).
We started from the two extreme
pipe dreams, $\mathbf{b}_{\mathrm{id}}$ and $\mathbf{b}_{w_0}$
(cf. \Cref{fig:height_function}).
\Cref{fig:mixing_comparison} shows the plot of $\ell(w)$ over $10^{9}$ steps
for each of the four rectangle size distributions, starting from both states.
The dashed line marks $\ell \approx 0.396\binom{n}{2}$,
the expected length of the optimal layered permutation
(\Cref{lem:layered_length_ratio}).
For all four distributions of the rectangle sizes, 
the chain started from $\mathbf{b}_{w_0}$ (red)
reaches this stationary range quickly, with no apparent difference between
the distributions.
The chain starting from $\mathbf{b}_{\mathrm{id}}$ (blue) rises much more slowly,
and systematically stays below the stationary value, even after $10^9$ steps.
For this initial condition,
the average permutation matrix and an example of an RBPD
(starting from $\mathbf{b}_{\mathrm{id}}$)
display characteristic ``stuck'' patterns which seem to prevent 
faster mixing,
see \Cref{fig:permuton_id} for
an illustration.
We conclude that \textbf{starting from $\mathbf{b}_{w_0}$
mixes more reliably}.

\begin{figure}[htpb]
\centering
\includegraphics[width=0.45\textwidth]{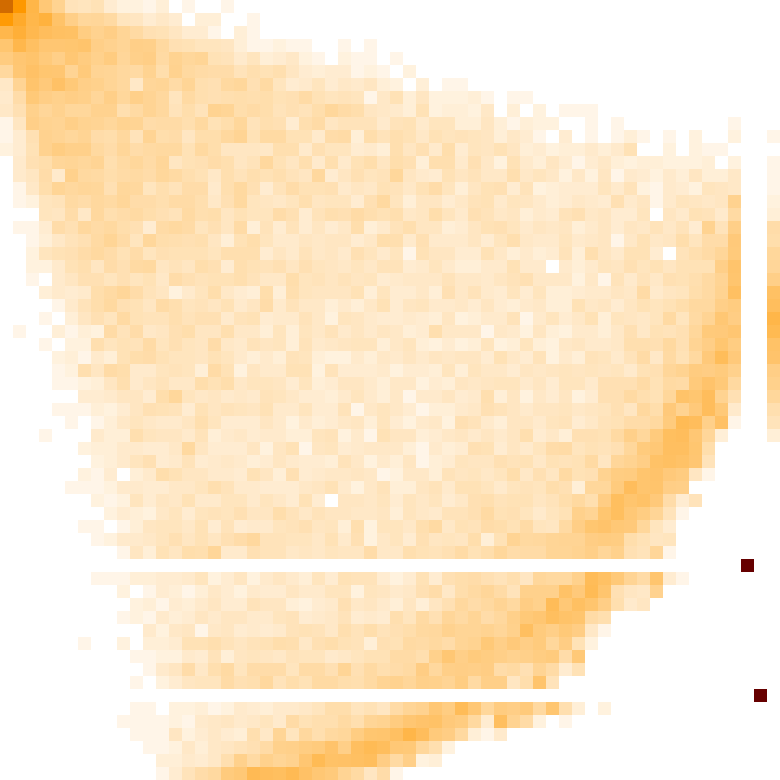}%
\hfill
\includegraphics[width=0.45\textwidth]{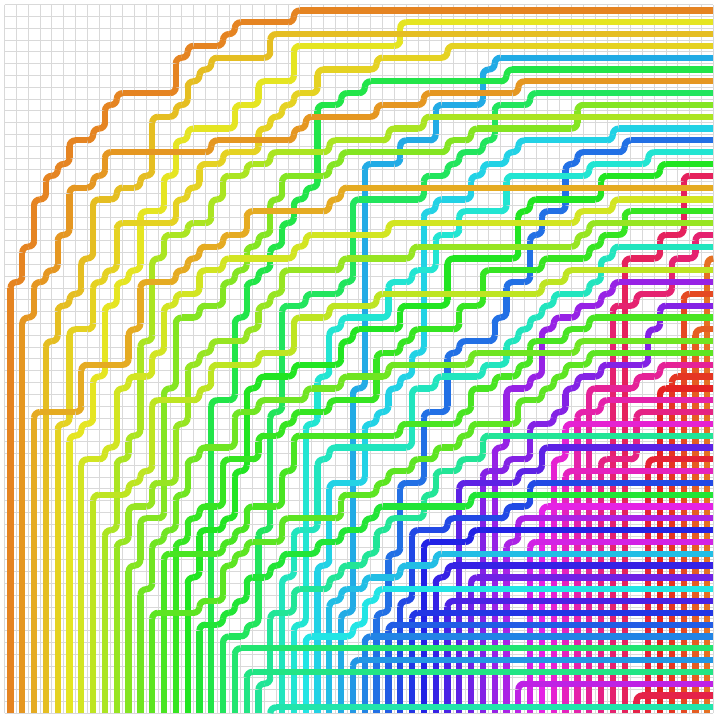}
\caption{A ``stuck'' chain started from $\mathbf{b}_{\mathrm{id}}$
at $n=60$ after $10^9$ steps.
Left: average permutation matrix
\eqref{eq:average_permutation_matrix}
over $500$ samples (thinning $10^6$ steps).
Right: a single RBPD with colored pipes.
Two values ($\sigma(44)=58$ and $\sigma(54)=59$)
are frozen across all $500$ samples (dark dots),
visible in the RBPD as vertical gaps of horizontal tiles
in the bottom-right corner. 
These samples were done with uniform droop proposals, but the same phenomenon appears with other distributions.}
\label{fig:permuton_id}
\end{figure}

Let us discuss the choice of the rectangle-size distribution for the
\emph{collection} phase (after burn-in).
We compared the four distributions 
(geometric, uniform, log-uniform, and reverse log-uniform)
by measuring the lag-1 autocorrelation
$\mathrm{Corr}(\ell(w_t), \ell(w_{t+1}))$ 
of the $\ell(w)$ statistic.
The geometric distribution (with parameter $1/2$, as 
above) achieves the lowest autocorrelation.
We attribute this to its acceptance rate:
geometric droops (small rectangles preferred)
are accepted 100 to 1000 times more often than for other 
distributions, which apparently leads to faster decorrelation between
successive samples.

For the simulations with $n=100$,
we use the geometric
rectangle-size distribution for droop proposals, both in burn-in and collection
phases.
We verified (by checking the equilibration of the $\ell(w)$ 
statistic) that a
burn-in of $10^{10}$ steps is sufficient for $n = 100$.
The $B = 10{,}000$ samples are then collected with thinning by $5 \times 10^8$
steps. We used 40 CPU cores and ran 40 Markov chains in parallel
on the Rivanna HPC cluster at the University of Virginia,
each collecting $250$ samples after an independent burn-in.
The total runtime for the burn-in and collection phases was about 3 hours.

Each sample records the boundary permutation $w$,
from which we accumulate
the \emph{average permutation matrix}
\begin{equation}
	\label{eq:average_permutation_matrix}
	(M_{ij})_{i,j=1}^n,\qquad 
	M_{ij} \coloneqq \frac{1}{B}\sum\nolimits_{t=1}^{B}\mathbf{1}_{w_t(i)=j},
\end{equation}
see \Cref{fig:permuton_n100} for the result.
The matrix $M$ serves as a histogram approximating the hypothetical
limiting \emph{permuton} \cite{hoppen2013limits}, \cite{grubel2023ranks}
of uniform random RBPDs:
if a limit shape exists as $n\to\infty$,
the rescaled matrix $M$ converges to the density of that permuton.

We also record the height function $h_t$ for each sample,
and compute the average height function 
\begin{equation}
	\label{eq:average_height_function}
	\bar{h}(x,y) \coloneqq \frac{1}{B}\sum\nolimits_{t=1}^{B} h_t(x,y),
	\qquad 0\le x,y \le n.
\end{equation}
From the height functions, we extract two quantities.
First, the discrete mixed derivative $\Delta_x \Delta_y \bar{h}$
outlines the ``\emph{liquid region}'' (called
by analogy with \cite{KOS2006}),
the part of the grid where the average height function is not linear
(where all six tile types coexist),
see \Cref{fig:overlay_n100}, left.
The complement to the liquid region consists of \emph{frozen regions}
consisting entirely of
empty, cross, or vertical/horizontal tiles.
Second, we record the height function fluctuations
$h_t - \bar{h}$, which are nontrivial inside the liquid region
and vanish in the frozen regions, see \Cref{fig:height_fluctuation_n100}.

\begin{figure}[htbp]
\centering
\includegraphics[width=0.4\textwidth]{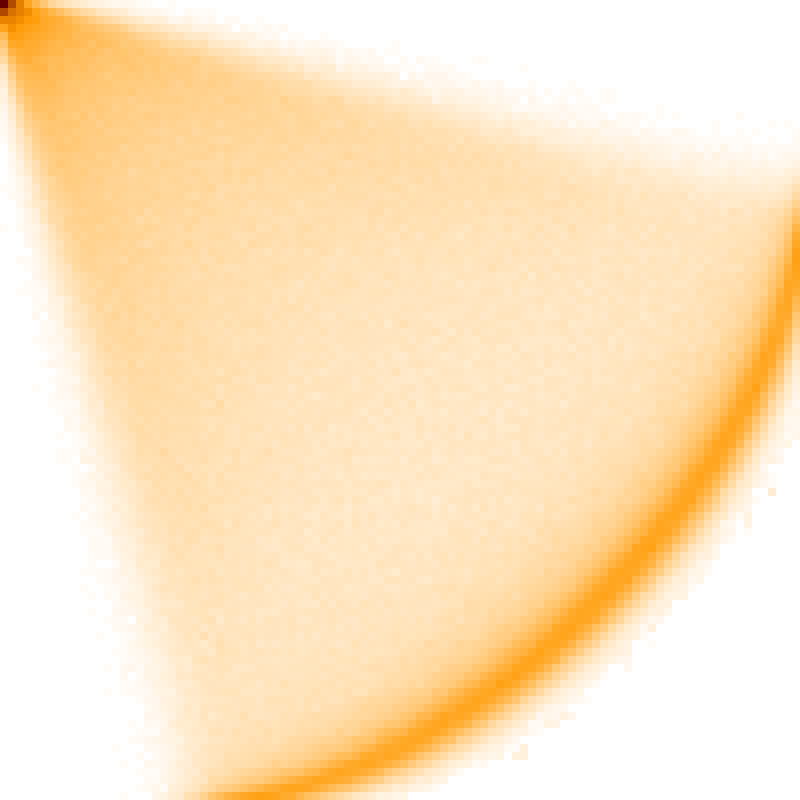}
\caption{Average permutation matrix \eqref{eq:average_permutation_matrix}
at $n=100$ from $10{,}000$ MCMC samples.
The support is contained inside a cone,
with a singular (delta-measure) component visible along the southeast boundary curve.}
\label{fig:permuton_n100}
\end{figure}

\begin{figure}[htbp]
\centering
\includegraphics[width=.4\textwidth]{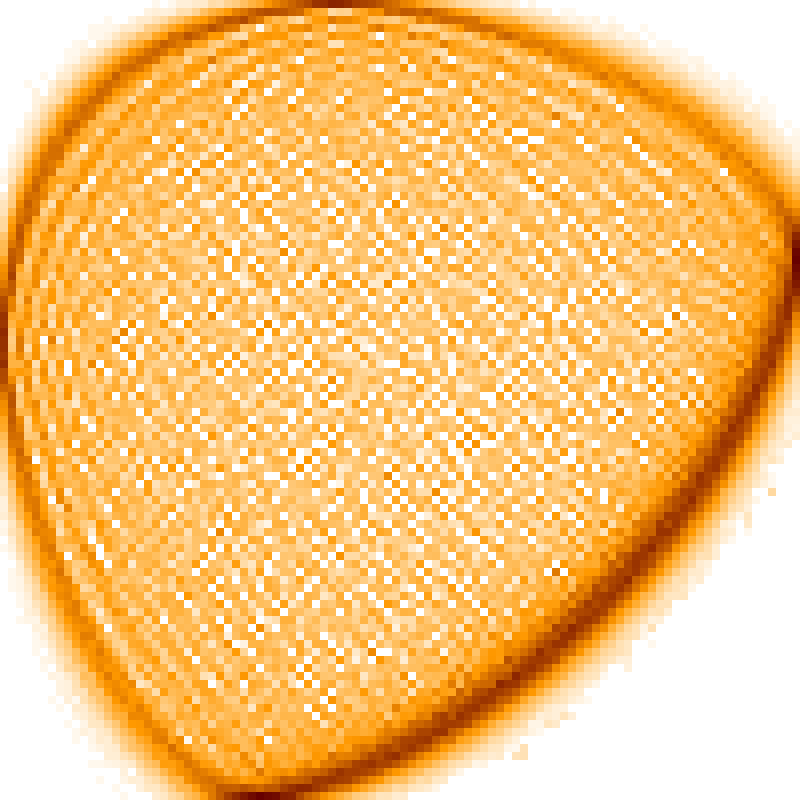}
\hspace{1cm}
\includegraphics[width=.4\textwidth]{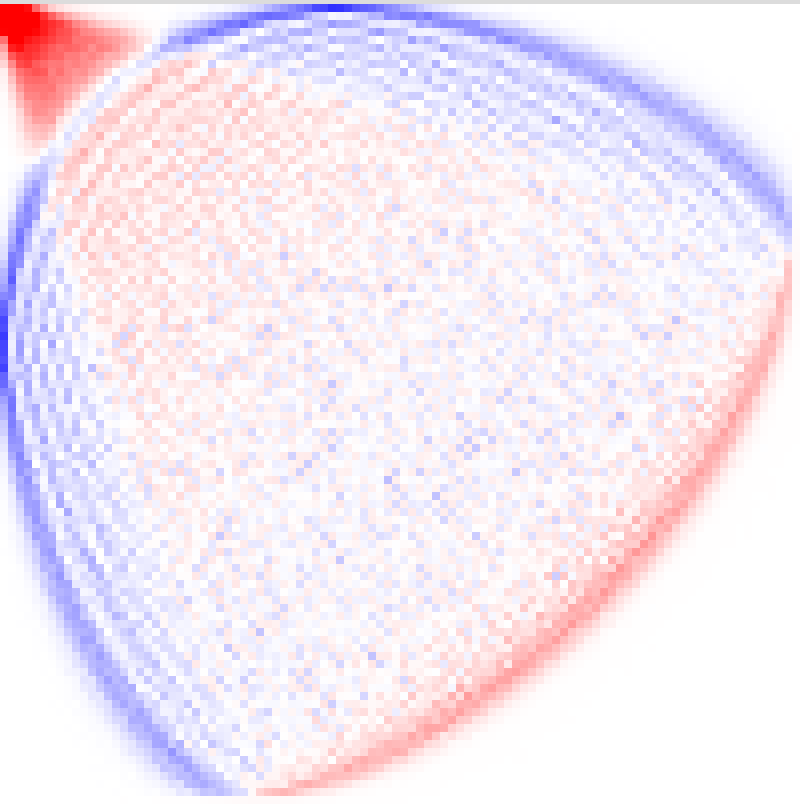}
\caption{Left: the discrete mixed derivative $\Delta_x \Delta_y \bar{h}$.
Right: plot of the difference of 
$\Delta_x \Delta_y \bar{h}$ and the average permutation matrix $M$
(normalized as probability distributions)
showing that the southeast boundary curves of $M$ and of the liquid region coincide
(see \Cref{conj:delta_boundary} and \Cref{prop:delta_boundary}).}
\label{fig:overlay_n100}
\end{figure}

\begin{figure}[htbp]
\centering
\includegraphics[width=0.45\textwidth]{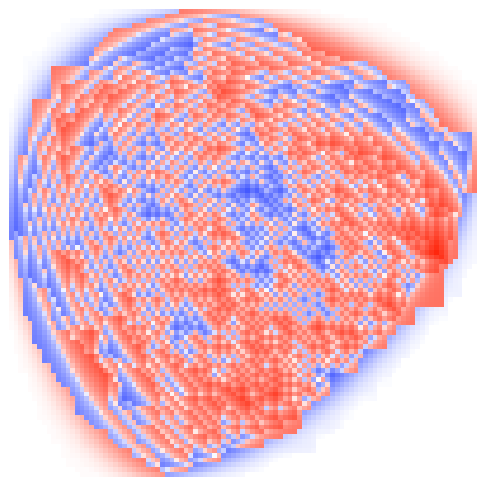}%
\hfill
\includegraphics[width=0.45\textwidth]{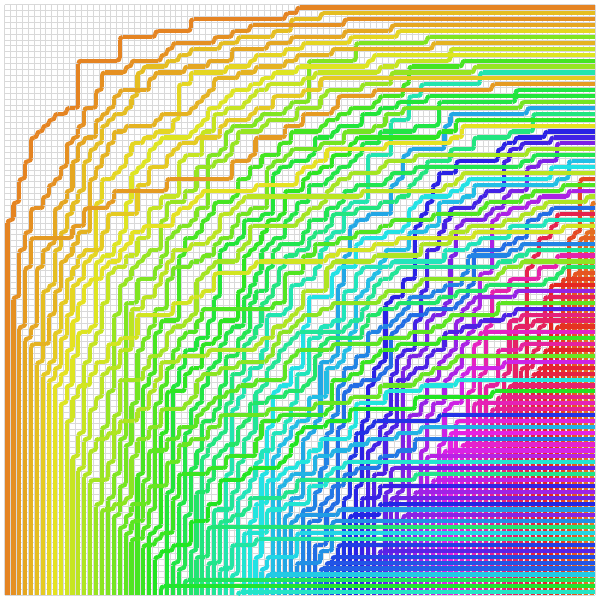}
\caption{A single MCMC sample at $n=100$.
Left: height function fluctuation $h_t - \bar{h}$;
red indicates $h_t > \bar{h}$, blue indicates $h_t < \bar{h}$.
Right: a random reduced bumpless pipe dream. 
One can see that the region adjacent to the southeast corner is 
saturated with cross tiles.
}
\label{fig:height_fluctuation_n100}
\end{figure}

\section{Conjectures and open problems}
\label{sec:further}

\subsection{Schubert measure on permutations and the limiting permuton}
\label{sec:permuton}

The uniform measure on RBPDs of size $n$
weights each permutation $w \in S_n$ proportionally
to its number of reduced bumpless pipe dreams,
which equals the principal specialization $\Upsilon_w$.
We may call the resulting probability measure on $S_n$,
\begin{equation}
\label{eq:schubert_measure}
	\Prob(w) = \frac{\Upsilon_w}{\sum_{v \in S_n} \Upsilon_v},
\end{equation}
the \emph{Schubert measure} on permutations.

A measure on permutations similar to \eqref{eq:schubert_measure} was studied in \cite{GrothendieckShenanigans2024},
when the Schubert polynomials are replaced by Grothendieck polynomials
with the Grothendieck parameter specialized at $\beta=1$.
As $n\to\infty$, these Grothendieck random permutations converge
\cite[Theorem~1.5]{GrothendieckShenanigans2024}
to a deterministic permuton
supported inside a cone, with a singular component concentrated on a quarter circle forming the southeast boundary of the support.
We refer to \cite{hoppen2013limits}, \cite{grubel2023ranks} for background on the permuton convergence
(which, in particular, is equivalent to the convergence of all finite pattern densities).

The average permutation matrix we observe for the Schubert measure is visually similar (\Cref{fig:permuton_n100}), motivating us to make the following conjecture:

\begin{conjecture}[Convergence to the Schubert permuton]
\label{conj:permuton}
The random permutations 
\eqref{eq:schubert_measure}
converge, as $n\to\infty$,
to a deterministic permuton $\mu$ on $[0,1]^2$
\emph{(}which we call the \emph{Schubert permuton)}
supported inside a cone and having a singular component along the southeast boundary curve.
\end{conjecture}
For the Schubert permuton, the boundary curve does not appear to be a quarter circle.

\subsection{Limit shape and arctic curve}
\label{sec:limit_shape}

The discrete mixed derivative $\Delta_x \Delta_y \bar{h}$
(\Cref{fig:overlay_n100}, left) 
and fluctuations of $h_t - \bar{h}$
(\Cref{fig:height_fluctuation_n100}, left)
detect the \emph{liquid region} of the 
uniformly random RBPD, where all six tile types coexist with positive density.
The liquid region is symmetric about the main diagonal
(reflecting the $w \mapsto w^{-1}$ symmetry of the model),
but does \emph{not} appear to have quarter-turn symmetry,
in contrast with the Grothendieck case \cite{GrothendieckShenanigans2024}
(equivalent to domino tilings of the Aztec diamond 
\cite{jockusch1998random}, \cite{cohn-elki-prop-96})
and the uniformly random six-vertex model with domain wall boundary conditions 
\cite{ColomoPronko2010limitshape},
\cite{Agg6V},
whose limit shapes are invariant under quarter-turn rotation of the square.

\begin{conjecture}[Limit shape and arctic curve]
\label{conj:delta_boundary}
As $n \to \infty$, the height function of a uniformly random
RBPD of size $n$ converges to a deterministic limit shape,
consisting of four frozen regions adjacent to the corners of the
domain (where the height function is linear)
and a liquid region (where it is curved),
separated by a deterministic \emph{arctic curve}.
The southeast arc of the RBPD arctic curve coincides with
the support of the singular component
of the Schubert permuton (\Cref{conj:permuton}).
\end{conjecture}

We prove the following 
coincidence of the southeast component of the arctic curve and
the singular component of the permuton, assuming that the 
frozen region in the RBPDs exists with an exponential rate of convergence:
\begin{proposition}
\label{prop:delta_boundary}
	Let $R_n \subset \{1,\ldots,n\}^2$ be a SE-justified
	region in the BPD grid,
	that is, if $(i,j) \in R_n$,
	then $(i',j') \in R_n$ for all $i' \ge i$ and $j' \ge j$.\footnote{The
	grid coordinates are as in \Cref{fig:rothe_diagram},
	with $(1,1)$ in the northwest corner and $(n,n)$ in the southeast corner.}
	Furthermore, assume that $R_n$ does not touch the north or west boundaries of the grid.
	Suppose that
	the RBPDs exhibit the \emph{frozen region:}
	for some constant $c > 0$ not depending on $n$,
	every tile $(i,j) \in R_n$ of a uniformly random RBPD of size $n$
	is a cross with probability at least $1 - e^{-cn}$.
	Then
	there exists a constant $c' > 0$ such that
	with probability $1-e^{-c'n}$ as $n\to\infty$, 
	the region $R_n$ in the permutation matrix of a 
	Schubert random permutation $w$ of size $n$
	does not contain any points, that is,
	$(w(k),\, k) \notin R_n$ for all $k \in \{1,\ldots,n\}$.
\end{proposition}
\begin{proof}
Condition on the 
event that every tile in $R_n$ is a cross (by a union bound over the $O(n^2)$ tiles in $R_n$, this has probability $1-e^{-c''n}$
for some $c'' > 0$).
Recall that the pipe $k$ enters the grid from the south at column $k$
and exits on the east boundary at row $w(k)$.

Assume first that $(n, k) \in R_n$.
Then the cells of $R_n$ in column $k$
form a contiguous block $\{i_k^*,\, i_k^*+1,\, \ldots,\, n\}$
for some $i_k^*$. 
Pipe $k$ enters from the south at $(n,k)$
and passes straight up through 
$R_n$, exiting column $k$ at row $\le i_k^* - 1$.
Since after that the pipe must continue traveling only north and east,
its exit row $w(k)$ must satisfy
$w(k) \le i_k^* - 1 < i_k^*$,
so $(w(k), k) \notin R_n$.

It remains to consider the case that $(n, k) \notin R_n$.
Then $(i,k) \notin R_n$ for all $i$.
So $(w(k), k) \notin R_n$ regardless of $w(k)$.
This completes the proof.
\end{proof}

Heuristically, in \Cref{fig:permuton_n100,fig:overlay_n100,fig:height_fluctuation_n100},
we see that some mass from the southeast frozen region in the RBPD 
grid is concentrated along the southeast boundary curve in the permuton,
but \Cref{prop:delta_boundary} does not rule out the possibility
that the permuton may be absolutely continuous with respect to the Lebesgue measure
in $[0,1]^2$.

\medskip

The height function fluctuations $h_t - \bar{h}$
are nontrivial inside the liquid region
and vanish in the frozen regions
(\Cref{fig:height_fluctuation_n100}).
For dimer models, such fluctuations converge
to the Gaussian Free Field (GFF) \cite{Kenyon2001GFF}, \cite{Petrov2012GFF}, \cite{bufetov2016fluctuations}, \cite{ChelkakLaslierRusskikh2021GFF}
(see \cite{Sheffield2007GFF} for the standard background on the GFF).
For the non-free-fermion six-vertex model
with domain wall boundary conditions,
the fluctuation structure is far less 
understood.\footnote{Even the full limit shape picture is
not fully developed, cf. 
\cite{ColomoPronko2010limitshape}, 
\cite{colomo2010arctic}, \cite{deGierKenyon2021limit}, \cite{Agg6V}.}

\begin{question}
\label{q:fluctuations}
Do the height function fluctuations of uniformly random RBPDs
(see \Cref{fig:height_fluctuation_n100})
converge to the GFF inside the liquid region?
\end{question}

In support of the GFF behavior, let us note that our MCMC samples at $n=100$
exhibit fluctuations of order $\sqrt{\log n}$:
the maximum of $|h_t - \bar{h}|$ is approximately $2.2$--$2.9$
across independent samples, consistent with $\sqrt{\log 100} \approx 2.15$.

\subsection{Merzon--Smirnov conjecture}
\label{sec:merzon_smirnov_further}

The Merzon--Smirnov conjecture (\Cref{conj:merzon_smirnov})
holds for $n \le 13$ by exhaustive search (\Cref{prop:full_search_13}),
extending prior verification for $n \le 10$ in \cite{merzon2016determinantal}.
At $n = 17$ the conjecture fails 
(\Cref{prop:merzon_smirnov_disproved}).
The counterexample $w^*$ \eqref{eq:counterexample}
differs from the optimal layered permutation $w(1,2,4,10)$
by a single adjacent transposition.
We computed $\Upsilon_w$ for all permutations within 
Cayley distance $4$
of the optimal layered permutation for each $n \le 16$,
and found no counterexamples;
at $n = 17$, the counterexample $w^*$ at Cayley distance $1$ is the unique immediate
neighbor of $w(1,2,4,10)$ that exceeds the layered maximum.
We also found counterexamples at $n = 18$, $19$, and $20$: in each case, the Cayley distance~$1$ neighborhood of the optimal layered permutation contains a unique permutation exceeding the layered maximum, obtained by swapping the last entry of the penultimate block with the first entry of the last block.
These are $s_7$ applied to $w(1,2,4,11)$ at $n=18$, $s_8$ applied to $w(1,2,5,11)$ at $n=19$, and $s_8$ applied to $w(1,2,5,12)$ at $n=20$, exceeding the layered maximum by about $5\%$, $9\%$, and $8\%$, respectively.
The $n=20$ computation required a 40-core cluster node (University of Virginia's Rivanna HPC) due to the size of the BFS frontiers.\footnote{While the benchmarks for layered permutations (\Cref{sec:layered_benchmark}) show the transition formula beating cotransition for $n=16,17$, for larger $n$ the caching required for this speedup becomes infeasible, so we default to cotransition and BFS rather than DFS; see \Cref{rmk:long_computation}.} 
However, applying the same swap to the optimal layered permutation for $n \le 16$ always \emph{decreases} $\Upsilon_w$, so the pattern does not extend downward.

Searching further from the optimal layered permutation
reveals even larger values.
The permutation $u^*$ \eqref{eq:counterexample2},
obtained from $w^*$ by additionally transposing the entries in positions~$4$ and~$17$,
exceeds the layered maximum by about $16\%$
(compared to the $7\%$ excess in
\Cref{prop:merzon_smirnov_disproved}).
The analogous permutation
\begin{equation}
\label{eq:more_excess}
(1,3,2,8,6,5,18,4,17,16,15,14,13,12,11,10,9,7) \in S_{18}
\end{equation}
exceeds the layered maximum by about $12\%$.
We illustrate the maximal layered and the newly discovered permutation in 
\Cref{fig:perm_comparison}.

\begin{figure}[ht]
\centering
\begin{tikzpicture}[scale=0.25]
  \begin{scope}
    \draw[gray!40, very thin] (0.5,0.5) grid (18.5,18.5);
    \foreach \v [count=\i] in {1,3,2,7,6,5,4,18,17,16,15,14,13,12,11,10,9,8} {
      \fill (\v,{19-\i}) circle (5pt);
    }
    \foreach \i in {1,2,3,18} {
      \node[below, font=\tiny] at (\i,0.2) {\i};
      \node[left, font=\tiny] at (0.2,{19-\i}) {\i};
    }
    \node[below=6pt, font=\small] at (9.5,0) {$w(1,2,4,11)$};
  \end{scope}
  \begin{scope}[shift={(22,0)}]
    \draw[gray!40, very thin] (0.5,0.5) grid (18.5,18.5);
    \foreach \v [count=\i] in {1,3,2,7,6,5,18,4,17,16,15,14,13,12,11,10,9,8} {
      \fill (\v,{19-\i}) circle (5pt);
    }
    \foreach \i/\v in {7/18,8/4} {
      \fill[red] (\v,{19-\i}) circle (5pt);
    }
    \foreach \i in {1,2,3,18} {
      \node[below, font=\tiny] at (\i,0.2) {\i};
    }
    \node[below=6pt, font=\small] at (9.5,0) {$w(1,2,4,11)\cdot s_7$};
  \end{scope}
  \begin{scope}[shift={(44,0)}]
    \draw[gray!40, very thin] (0.5,0.5) grid (18.5,18.5);
    \foreach \v [count=\i] in {1,3,2,8,6,5,18,4,17,16,15,14,13,12,11,10,9,7} {
      \fill (\v,{19-\i}) circle (5pt);
    }
    \foreach \i/\v in {4/8,7/18,8/4,18/7} {
      \fill[red] (\v,{19-\i}) circle (5pt);
    }
    \foreach \i in {1,2,3,18} {
      \node[below, font=\tiny] at (\i,0.2) {\i};
    }
    \node[below=6pt, font=\small] at (9.5,0) {Permutation \eqref{eq:more_excess}};
  \end{scope}
\end{tikzpicture}
\caption{Permutation matrices for the $n=18$
examples: the optimal layered permutation $w(1,2,4,11)$ (left),
and two permutations exceeding the layered maximum.}
\label{fig:perm_comparison}
\end{figure}
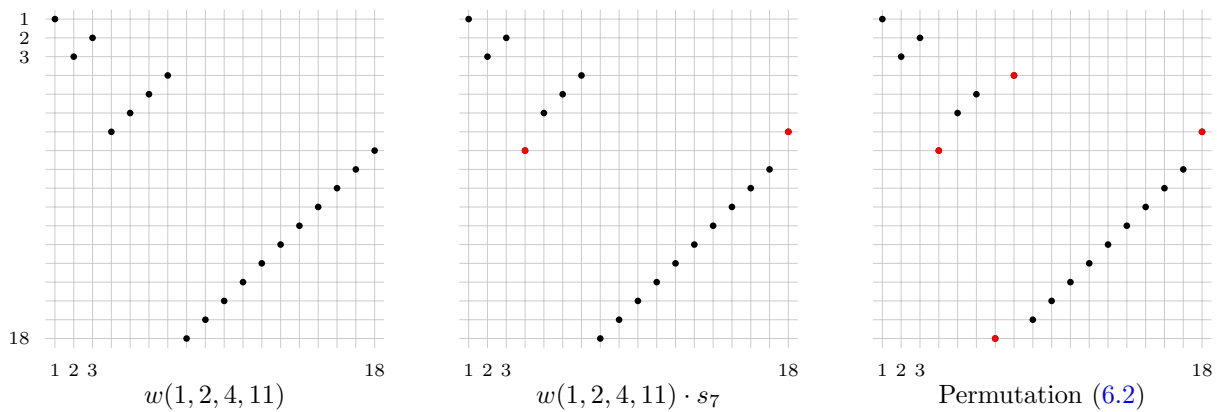

\begin{remark} It now becomes an even more challenging problem to characterize the permutations maximizing $\Upsilon_w$ for general $n$: the newly found examples show that the characteristics would likely not be structural (e.g. pattern avoidance, block structure etc).
\end{remark}

\subsection{Bounds and asymptotics}
\label{sec:asymp}

Despite the failure of Conjecture~\ref{conj:merzon_smirnov}, we expect that the asymptotic growth rate in Stanley's problem \eqref{eq:stanley_question} is determined by layered permutations:
\begin{conjecture}
\label{conj:layered_asymptotic}
The maximum of $\Upsilon_w$ over layered permutations has the 
same exponential growth rate as 
the maximum over the whole $S_n$, that is,
\begin{equation*}
	\lim_{n\to \infty} \frac{1}{n^2} \log_2 \,\max_{w \in S_n} \Upsilon_w
	\,=\,
	\lim_{n\to \infty} \frac{1}{n^2} \log_2 \,\max_{w \text{ layered}} \Upsilon_w
	\,\approx\, 0.29,
\end{equation*}
where the layered limit was computed in \cite{MoralesPakPanova2019}.
In other words, any potential improvement over layered permutations is subexponential in $n^2$.
\end{conjecture}

In support of this we observe the asymptotic behavior of nearby permutations. 

\begin{proposition}[Bounds]\label{Prop:bounds}
    Suppose that $u=w \cdot s_i$ or $u=s_i \cdot w$ with $\ell(u)=\ell(w)+1$. Then 
    \begin{equation}\label{eq:bounds}
        \frac{1}{n-i} \leq \frac{ \Upsilon_u}{\Upsilon_w} \leq i.
    \end{equation}
\end{proposition}

\begin{proof}
    Suppose that $u = w \cdot s_i$. The case $u = s_i \cdot w$ follows from it by observing that 
    $u^{-1} = w^{-1} s_i $ and that $\Upsilon_u = \Upsilon_{u^{-1}}$ for every permutation.
    
   For the lower bound, we use the divided difference formula, $\mathfrak{S}_w(x_1,\ldots,x_n) = \partial_i \mathfrak{S}_u(x_1,\ldots,x_n)$.  (See Appendix~\ref{app:schubert_background}.)  Every monomial of $\mathfrak{S}_u$ is of the form $x_1^{a_1}\cdots x_{n-1}^{a_{n-1}}$ with $a_i \leq n-i$ (as can be easily seen from the PD interpretation).  We have 
   $$\partial_i {\mathbf x}^{\mathbf a} = \begin{cases} \sum_{j=0}^{a_i-a_{i+1}-1} x_i^{ a_i -j-1} x_{i+1}^{a_{i+1}+j}\prod_{k \neq i,i+1} x_k^{a_k}, & \text{ for }a_i>a_{i+1},\\
   -\sum_{j=0}^{a_{i+1}-a_i-1} x_i^{ a_i -j-1} x_{i+1}^{a_{i+1}+j}\prod_{k \neq i,i+1} x_k^{a_k}, & \text{ for }a_i<a_{i+1},\\
   0, & \text{ for }a_i=a_{i+1}.\end{cases}
   $$ 
   Setting all $x$ variables equal to $1$, we find that the specialization of $\partial_i\mathbf{x}^{\mathbf{a}}$ is bounded above by $|a_i - a_{i+1}| \leq n-i$.
  Applying this bound to every monomial in $\mathfrak{S}_u$, we obtain
  $$\Upsilon_w \leq (n-i)\, \Upsilon_u,$$
  which gives the desired lower bound.

    To establish the upper bound, we use Monk's formula for multiplying Schubert polynomials; see \Cref{eq:monk}.  Evaluating at $1$, this says
    \[
    i\,\Upsilon_w = \sum_{r\leq i<s}\Upsilon_{w\cdot (r,s)},
    \]
    the sum over transpositions $(r,s)$ such that $\ell(w\cdot (r,s)) = \ell(w)+1$.  Evidently from the assumption that $u = w\cdot (i,i+1)$ and $\ell(u)=\ell(w)+1$, $u$ is among the permutations appearing on the RHS.  This gives
    \[
    i\,\Upsilon_w \geq \Upsilon_u
    \]
    as claimed.
\end{proof}

\begin{corollary}\label{cor:bounds} 
    Suppose that two permutations $u$ and $w$ differ by at most a linear number of adjacent transpositions, i.e. 
    $ u = \sigma w \pi$, where $\ell(\sigma), \ell(\pi)\leq a n$ for some constant $a$. 
    Then 
    $$ 2^{-2an\log(n)} \leq \frac{ \Upsilon_u}{\Upsilon_w} \leq 2^{2an\log(n)}$$
\end{corollary}
\begin{proof}
Follows by presenting $\sigma$ and $\pi$ as products of at most $an$ simple transpositions each, and repeatedly applying \Cref{Prop:bounds}. 
\end{proof}

\begin{remark}
In particular, \Cref{cor:bounds} implies that if a permutation $w$ is close to maximal, i.e. $\Upsilon_w = 2^{cn^2 + o(n^2)}$, then a permutation $u$ differing from $w$ by a linear number of transpositions will have the same leading term asymptotics. Thus, the counterexamples found in~\Cref{prop:merzon_smirnov_disproved} do not contradict the conjectured leading order asymptotics. 
\end{remark}

To characterize the maximal permutations, we discuss some consequences of the descent formula (\Cref{thm:descent}).

\begin{proposition}
\label{cor:descent}
        Suppose that $w \in S_n$ is such that $\Upsilon_w$ is maximal among all specializations for permutations in $S_n$. Then $\maj(w) \geq \ell(w)$,
				where $\maj(w) \coloneqq \sum_{i \in \Des(w)} i$ is the major index of the permutation~$w$.
\end{proposition}

\begin{proof}
    We have that $\Upsilon_{w\cdot s_i} \leq \Upsilon_w$, so \eqref{eq:descent} gives
    \begin{equation*}
			\Upsilon_w = \sum_{i \in \Des(w)}
			\frac{i}{\ell(w)} \cdot \Upsilon_{w \cdot s_i} \leq   \sum_{i \in \Des(w)}
			\frac{i}{\ell(w)} \cdot \Upsilon_w.
		\end{equation*}
    After cancellation, we get $\ell(w) \leq \sum_{i\in \Des(w)} i$, as desired.
\end{proof}
\begin{remark}
	One can directly check that
	\Cref{cor:descent} holds
	for all layered
	permutations.
\end{remark}

\subsection{Connectivity and exact sampling}
\label{sec:connectivity_open}

We bypassed the issue of connectivity of RBPDs under local flips
in $2\times 2$ windows by adding droop moves to the MCMC sampler,
but the question of connectivity under local flips alone
remains open (\Cref{conj:rbpd_connectivity}).
Neither $2\times 2$ flips nor droops/undroops seem to be 
appropriate ingredients for an exact
sampling algorithm with CFTP \cite{ProppWilsonCP},
and so the question of exact sampling of uniform RBPDs also remains open.
We remark that sometimes exact sampling can be achieved by a non-local Markov chain
utilizing algebraic/combinatorial structure of the model,
for example, in shuffling algorithms for domino \cite{propp2003generalized}
and lozenge \cite{borodin-gr2009q} tilings.

\subsection{Mixing time}
\label{sec:mixing_open}

We have no rigorous mixing time bounds for the MCMC chain
of \Cref{sec:mcmc_sampler}.
Our experiments (\Cref{fig:mixing_comparison})
suggest that with sufficiently large rectangular moves
(such as the reverse log-uniform distribution on rectangle sizes),
the chain mixes from all starting states in polynomial time.
On the other hand,
the chain restricted to $2\times 2$ flips alone
appears to mix exponentially slowly (due to ``stuck'' states as in \Cref{fig:permuton_id}).
For comparison, 
Glauber dynamics for the six-vertex model
(which performs the same $2\times 2$ flips as in \Cref{sec:local_flips}, but with probabilities depending 
on the six-vertex weights)
mixes exponentially slowly 
for certain parameter regimes \cite{FahrbachRandall2019}.
Establishing a polynomial/exponential distinction rigorously
for the RBPD chain is
an important open problem.

\subsection{Other specializations and measures}
\label{sec:other_specializations}

One can extend our methods beyond the principal specialization
$\Upsilon_w = \mathfrak{S}_w(1^n)$
in several directions.
The $q$-specialization $\mathfrak{S}_w(1, q, q^2, \ldots)$
refines $\Upsilon_w$, and one can readily assign a $q$-weight to each RBPD,
and perform a similar MCMC sampling investigation.
The double Schubert polynomials $\mathfrak{S}_w(x;\, y)$
(Appendix~\ref{app:schubert_background})
also enjoy
transition and cotransition formulas
\cite{Knutson2019cotransition}, \cite{LascouxSchutzenberger1985LR}.
For them, one can adapt our algorithms 
described in \Cref{sec:computation}
and compute $\mathfrak{S}_w(x;\, y)$ at various specialization points.
The Grothendieck polynomials
$\mathfrak{G}^{(\beta)}_w(x)$
deform the Schubert polynomials by adding a parameter $\beta$
(setting $\beta=0$ recovers the Schubert case).
For all $\beta>0$, the bumpless pipe dreams are allowed to be non-reduced,
but their enumeration involves the factors $\beta^{-\ell(w)}$. 
For $\beta=1$, these factors disappear, and the bumpless pipe dream picture
becomes equivalent to domino tilings of the Aztec diamond \cite[Section~6]{GrothendieckShenanigans2024}.
The permuton picture was also developed for the $\beta=1$ case in \cite{GrothendieckShenanigans2024}.
It would be interesting to explore MCMC sampling and principal specializations
for all $0<\beta<1$.

Another natural relative of the Schubert measure is the
\emph{symmetrized product measure}
$\Prob(w) \propto \Upsilon_w \cdot \Upsilon_{w_0 w}$.
Its expected length can be computed exactly from the Cauchy identity
for Schubert polynomials \cite{MacdonaldSchubertBook}:
\begin{equation}
\label{eq:cauchy_schubert}
\sum_{w \in S_n} \mathfrak{S}_w(x)\,\mathfrak{S}_{w_0 w}(y)
= \prod_{i+j \le n}(x_i + y_j).
\end{equation}
Setting $x_i = t$ and $y_j = 1$ and using
$\mathfrak{S}_w(t^n) = t^{\ell(w)}\Upsilon_w$
gives
$\sum_w t^{\ell(w)}\Upsilon_w\,\Upsilon_{w_0 w} = (1+t)^{\binom{n}{2}}$.
Taking $\frac{d}{dt}\big|_{t=1}$ of both sides, divided by
the value at $t=1$, yields
$\mathbb{E}[\ell(w)] = \binom{n}{2}/2$.
This contrasts with the Schubert measure \eqref{eq:schubert_measure},
which (conjecturally) concentrates on permutations of length $\approx 0.396\binom{n}{2}$, see
\Cref{lem:layered_length_ratio}.

\subsection{Code availability}
\label{sec:code}

The code accompanying this paper is available at
\url{https://github.com/lenis2000/schubert-computations-sampling}.
The repository includes:
\begin{enumerate}[$\bullet$]
\item \texttt{schubert.cpp} ---
descent, cotransition, and transition formulas
for computing $\Upsilon_w$, together with exact and heuristic max search.
\item \texttt{bpd\_mcmc.cpp} ---
MCMC sampler for reduced bumpless pipe dreams.
\item \texttt{bpd\_cftp\_sampler.cpp} ---
backward-CFTP sampler used for the negative results of \Cref{sec:cftp_failure}.
\item Diagnostic and validation tools
for CFTP failure analysis and RBPD connectivity.
\end{enumerate}

\appendix
\section{Background on Schubert polynomials}
\label{app:schubert_background}

We collect here the standard definitions underlying the recurrences
used in \Cref{sec:descent,sec:cotrans}.
Standard references include
Macdonald \cite{MacdonaldSchubertBook} and Manivel \cite{ManivelBook}.

For $1 \le i \le n-1$, the \emph{divided difference operator}
$\partial_i$ acts on polynomials $f \in \mathbb{Z}[x_1,\ldots,x_n]$ by
\begin{equation}
	\label{eq:divided_diff}
	\partial_i f \coloneqq \frac{f - s_i \cdot f}{x_i - x_{i+1}},
\end{equation}
where $s_i$ acts on $f$ by permuting $x_i$ and $x_{i+1}$.
These operators satisfy the nilCoxeter relations:
$\partial_i^2 = 0$,
$\partial_i \partial_j = \partial_j \partial_i$ for $|i-j| \ge 2$,
and $\partial_i \partial_{i+1} \partial_i = \partial_{i+1} \partial_i \partial_{i+1}$.

Schubert polynomials are recursively determined by the following conditions:
\begin{enumerate}[$\bullet$]
	\item For the longest permutation $w_0 = (n, n-1, \ldots, 2, 1)$,
		we have
		$\mathfrak{S}_{w_0} = x_1^{n-1}\ssp x_2^{n-2} \cdots x_{n-1}$.
	\item For all $w \in S_n$ and $i = 1, \ldots, n-1$ such that
		$\ell(w s_i) = \ell(w) + 1$,
		we have
		$\mathfrak{S}_w = \partial_i\ssp \mathfrak{S}_{w s_i}$.
\end{enumerate}
This recursive definition was introduced by
Bernstein, Gelfand, and Gelfand \cite{bernstein1973schubert}
and Demazure \cite{demazure1974desingularisation}
in the context of Schubert classes in the cohomology of flag varieties,
and made explicit by
Lascoux and Sch\"utzenberger \cite{LascouxSchutzenberger1982Schubert}.

\medskip
Lascoux and Sch\"utzenberger  \cite{LascouxSchutzenberger1985Interpolation} also introduced \emph{double Schubert polynomials}
\begin{equation*}
	\mathfrak{S}_w(x_1, \ldots, x_n;\, y_1, \ldots, y_n),
\end{equation*}
depending on two sets of variables.
They are defined by the same divided difference recursion \eqref{eq:divided_diff},
with $\partial_i$ acting on the $x$-variables only, and the modified base case
\[
	\mathfrak{S}_{w_0}(x;\, y) = \prod_{i+j \le n} (x_i - y_j).
\]
The ordinary Schubert polynomial is recovered by setting $y = 0$:
$\mathfrak{S}_w(x) = \mathfrak{S}_w(x;\, 0, \ldots, 0)$.
The transition formula of Lascoux-Sch\"utzenberger \cite{LascouxSchutzenberger1985LR} (\Cref{thm:trans}, see also \cite{KohnertVeigneau1997Schubert,Weigandt2020_bumpless}) and the cotransition formula of Knutson \cite{Knutson2019cotransition}
(\Cref{thm:cotrans}) hold more generally for double Schubert polynomials;
our specialization $\Upsilon_w = \mathfrak{S}_w(1^n)$
corresponds to $x_i = 1$, $y_j = 0$ for all $i, j$.
The transition and cotransition formulas follow quickly from {\it Monk's formula}, which describes how to multiply by a linear Schubert polynomial $\mathfrak{S}_{s_k}$.  Restricting attention to single Schubert polynomials, the formula is
\begin{equation}\label{eq:monk}
  (x_1+\cdots+x_k)\cdot \mathfrak{S}_w = \sum_{i\leq k<j} \mathfrak{S}_{w\cdot(i,j)},
\end{equation}
the sum over transpositions $(i,j)$ such that $\ell( w\cdot (i,j) ) = \ell(w)+1$.  This is equivalent to
\begin{equation}
    x_r \mathfrak{S}_w = -\sum_{i<r} \mathfrak{S}_{w\cdot(i,r)} + \sum_{j>r}\mathfrak{S}_{w\cdot(r,j)},
\end{equation}
where both sums run over only those terms of length equal to $\ell(w)+1$.  Transition is the case where the second sum consists of exactly one term (move the first sum to the other side to obtain a positive formula); this can occur for several choices of $r$.  Cotransition is the case where the first sum is empty; this only occurs for $r$ as in Theorem~\ref{thm:cotrans}.

\medskip
Schubert polynomials are independent of the ambient symmetric group: if
$\iota\colon S_n \hookrightarrow S_{n+1}$ denotes the natural inclusion
(fixing $n+1$), then
\begin{equation}
	\label{eq:stability}
	\mathfrak{S}_w(x_1, \ldots, x_n) = \mathfrak{S}_{\iota(w)}(x_1, \ldots, x_{n+1}).
\end{equation}
In particular, $\mathfrak{S}_{\iota(w)}$ does not depend on $x_{n+1}$,
so the principal specialization satisfies
$\Upsilon_w = \mathfrak{S}_w(1^n) = \mathfrak{S}_{\iota(w)}(1^{n+1})$.
This means $\Upsilon_w$ depends only on the
permutation $w$ belonging to the infinite symmetric group $S_\infty$
(i.e., the inductive limit of the $S_n$'s),
provided that the number of ones
in the Schubert polynomial is large enough to accommodate $w$.


\bibliography{bib}

\begin{bibdiv}
\begin{biblist}

\bib{Agg6V}{article}{
      author={Aggarwal, A.},
       title={{Arctic Boundaries of the Ice Model on Three-Bundle Domains}},
        date={2020},
     journal={Invent. Math.},
      volume={220},
       pages={611\ndash 671},
        note={arXiv:1812.03847 [math.PR]},
}

\bib{BjornerBrenti2005}{book}{
      author={Bj{\"o}rner, A.},
      author={Brenti, F.},
       title={{Combinatorics of Coxeter Groups}},
      series={Grad. Texts Math.},
   publisher={Springer},
     address={Berlin, Heidelberg},
        date={2005},
      volume={231},
}

\bib{BilleyBergeron}{article}{
      author={Bergeron, N.},
      author={Billey, S.},
       title={{R{C}-graphs and {Schubert} polynomials}},
        date={1993},
        ISSN={1058-6458,1944-950X},
     journal={Experiment. Math.},
      volume={2},
      number={4},
       pages={257\ndash 269},
         url={http://projecteuclid.org/euclid.em/1048516036},
      review={\MR{1281474}},
}

\bib{bufetov2016fluctuations}{article}{
      author={Bufetov, A.},
      author={Gorin, V.},
       title={{Fluctuations of particle systems determined by Schur generating
  functions}},
        date={2018},
     journal={Adv. Math.},
      volume={338},
       pages={702\ndash 781},
        note={arXiv:1604.01110 [math.PR]},
}

\bib{bernstein1973schubert}{article}{
      author={Bernstein, I.N.},
      author={Gelfand, I.M.},
      author={Gelfand, S.I.},
       title={{Schubert cells, and the cohomology of the spaces G/P}},
        date={1973},
     journal={Russian Math. Surveys},
      volume={28},
      number={3},
       pages={1\ndash 26},
}

\bib{BilleyGaoPawlowski2025}{article}{
      author={Billey, S.~C.},
      author={Gao, Y.},
      author={Pawlowski, B.},
       title={{Introduction to the Cohomology of the Flag Variety}},
        date={2025},
     journal={arXiv preprint},
        note={arXiv:2506.21064 [math.CO]},
}

\bib{borodin-gr2009q}{article}{
      author={Borodin, A.},
      author={Gorin, V.},
      author={Rains, E.},
       title={{q-Distributions on boxed plane partitions}},
        date={2010},
     journal={Selecta Math.},
      volume={16},
      number={4},
       pages={731\ndash 789},
        note={arXiv:0905.0679 [math-ph]},
}

\bib{BilleyHolroydYoung2019}{article}{
      author={Billey, S.~C.},
      author={Holroyd, A.~E.},
      author={Young, B.},
       title={{A bijective proof of {M}acdonald's reduced word formula}},
        date={2019},
     journal={Algebr. Comb.},
      volume={2},
      number={2},
       pages={217\ndash 248},
        note={arXiv:1702.02936 [math.CO]},
}

\bib{Bressoud1999}{book}{
      author={Bressoud, D.~M.},
       title={{Proofs and Confirmations: The Story of the Alternating-Sign
  Matrix Conjecture}},
   publisher={Cambridge University Press},
        date={1999},
        ISBN={9780511613449},
}

\bib{cohn-elki-prop-96}{article}{
      author={Cohn, H.},
      author={Elkies, N.},
      author={Propp, J.},
       title={{Local statistics for random domino tilings of the Aztec
  diamond}},
        date={1996},
     journal={Duke Math. J.},
      volume={85},
      number={1},
       pages={117\ndash 166},
        note={arXiv:math/0008243 [math.CO]},
}

\bib{ChelkakLaslierRusskikh2021GFF}{article}{
      author={Chelkak, D.},
      author={Laslier, B.},
      author={Russkikh, M.},
       title={{Bipartite dimer model: perfect t-embeddings and Lorentz-minimal
  surfaces}},
        date={2021},
     journal={arXiv preprint},
        note={arXiv:2109.06272 [math-ph]},
}

\bib{colomo2010arctic}{article}{
      author={Colomo, F.},
      author={Pronko, A.},
       title={The arctic curve of the domain-wall six-vertex model},
        date={2010},
     journal={J. Stat. Phys},
      volume={138},
      number={4-5},
       pages={662\ndash 700},
        note={arXiv:0907.1264 [math-ph]},
}

\bib{ColomoPronko2010limitshape}{article}{
      author={Colomo, F.},
      author={Pronko, A.~G.},
       title={{The limit shape of large alternating sign matrices}},
        date={2010},
     journal={SIAM J. Discrete Math.},
      volume={24},
      number={4},
       pages={1558\ndash 1571},
        note={arXiv:0803.2697 [math-ph]},
}

\bib{demazure1974desingularisation}{article}{
      author={Demazure, M.},
       title={{D{\'e}singularisation des vari{\'e}t{\'e}s de Schubert
  g{\'e}n{\'e}ralis{\'e}es}},
        date={1974},
     journal={Ann. Sci. {\'E}cole Norm. Sup{\'e}r.},
      volume={7},
      number={1},
       pages={53\ndash 88},
}

\bib{deGierKenyon2021limit}{article}{
      author={de~Gier, J.},
      author={Kenyon, R.},
      author={Watson, S.S.},
       title={Limit shapes for the asymmetric five vertex model},
        date={2021},
     journal={Commun. Math. Phys.},
      volume={385},
       pages={793\ndash 836},
        note={arXiv:1812.11934 [math.PR]},
}

\bib{FanGuoSun2018Bumpless}{article}{
      author={Fan, N.J.Y.},
      author={Guo, P.L.},
      author={Sun, S.C.C.},
       title={{Bumpless pipedreams, reduced word tableaux and {Stanley}
  symmetric functions}},
        date={2018},
     journal={arXiv preprint},
        note={arXiv:1810.11916 [math.CO]},
}

\bib{FK}{article}{
      author={Fomin, S.},
      author={Kirillov, A.N.},
       title={{Reduced words and plane partitions}},
        date={1997},
        ISSN={0925-9899,1572-9192},
     journal={J. Algebraic Combin.},
      volume={6},
      number={4},
       pages={311\ndash 319},
         url={https://doi.org/10.1023/A:1008694825493},
      review={\MR{1471891}},
}

\bib{FahrbachRandall2019}{inproceedings}{
      author={Fahrbach, M.},
      author={Randall, D.},
       title={{Slow Mixing of {Glauber} Dynamics for the Six-Vertex Model in
  the Ordered Phases}},
        date={2019},
   booktitle={Approx/random 2019},
      series={LIPIcs},
      volume={145},
       pages={37:1\ndash 37:20},
        note={arXiv:1904.01495 [cs.DS]},
}

\bib{FominStanley1994}{article}{
      author={Fomin, S.},
      author={Stanley, R.~P.},
       title={{Schubert polynomials and the nil{C}oxeter algebra}},
        date={1994},
     journal={Adv. Math.},
      volume={103},
      number={2},
       pages={196\ndash 207},
}

\bib{grubel2023ranks}{article}{
      author={Gruebel, R.},
       title={{Ranks, copulas, and permutons}},
        date={2024},
     journal={Metrika},
      volume={87},
       pages={155\ndash 182},
        note={arXiv:2206.12153 [math.ST]},
}

\bib{hoppen2013limits}{article}{
      author={Hoppen, C.},
      author={Kohayakawa, Y.},
      author={Moreira, C.~G.},
      author={R{\'a}th, B.},
      author={Sampaio, R.~M.},
       title={{Limits of permutation sequences}},
        date={2013},
     journal={J. Comb. Theory Ser. B},
      volume={103},
      number={1},
       pages={93\ndash 113},
        note={arXiv:1103.5844 [math.CO]},
}

\bib{jockusch1998random}{article}{
      author={Jockusch, W.},
      author={Propp, J.},
      author={Shor, P.},
       title={{Random Domino Tilings and the Arctic Circle Theorem}},
        date={1998},
     journal={arXiv preprint},
        note={arXiv:math/9801068 [math.CO]},
}

\bib{Kenyon2001GFF}{article}{
      author={Kenyon, R.},
       title={{Dominos and the Gaussian Free Field}},
        date={2001},
     journal={Ann. Probab.},
      volume={29},
      number={3},
       pages={1128\ndash 1137},
        note={arXiv:math-ph/0002027},
}

\bib{Knutson2019cotransition}{article}{
      author={Knutson, A.},
       title={{{Schubert} polynomials, pipe dreams, equivariant classes, and a
  co-transition formula}},
        date={2019},
     journal={arXiv preprint},
        note={arXiv:1909.13777 [math.CO]},
}

\bib{KOS2006}{article}{
      author={Kenyon, R.},
      author={Okounkov, A.},
      author={Sheffield, S.},
       title={Dimers and amoebae},
        date={2006},
     journal={Ann. Math.},
      volume={163},
       pages={1019\ndash 1056},
        note={arXiv:math-ph/0311005},
}

\bib{kuperberg1996another}{article}{
      author={Kuperberg, G.},
       title={Another proof of the alternative-sign matrix conjecture},
        date={1996},
     journal={Intern. Math. Research Notices},
      volume={1996},
      number={3},
       pages={139\ndash 150},
}

\bib{KohnertVeigneau1997Schubert}{article}{
      author={Kohnert, A.},
      author={Veigneau, S.},
       title={{Using {Schubert} basis to compute with multivariate
  polynomials}},
        date={1997},
     journal={Adv. Appl. Math.},
      volume={19},
      number={1},
       pages={45\ndash 60},
}

\bib{Lascoux02ice}{unpublished}{
      author={Lascoux, A.},
       title={{Chern and Yang through ice}},
        date={2002},
        note={Unpublished manuscript},
}

\bib{LamLeeShimozono2021BPD}{article}{
      author={Lam, T.},
      author={Lee, S.~J.},
      author={Shimozono, M.},
       title={{Back stable {Schubert} calculus}},
        date={2021},
     journal={Compos. Math.},
      volume={157},
      number={5},
       pages={883\ndash 962},
        note={arXiv:1806.11233 [math.CO]},
}

\bib{LascouxSchutzenberger1982Schubert}{article}{
      author={Lascoux, A.},
      author={Sch\"utzenberger, M.-P.},
       title={{Polyn\^omes de {Schubert}}},
        date={1982},
     journal={C. R. Acad. Sci. Paris S\'er. I Math.},
      volume={294},
       pages={447\ndash 450},
}

\bib{LascouxSchutzenberger1985Interpolation}{incollection}{
      author={Lascoux, A.},
      author={Sch\"utzenberger, M.-P.},
       title={{Interpolation de {Newton} \`a plusieurs variables}},
        date={1985},
   booktitle={S\'eminaire d'alg\`ebre {P}aul {D}ubreil et {M}arie-{P}aule
  {M}alliavin, 36\`eme ann\'ee (paris, 1983--1984)},
      series={Lecture Notes in Math.},
      volume={1146},
   publisher={Springer},
     address={Berlin},
       pages={161\ndash 175},
}

\bib{LascouxSchutzenberger1985LR}{article}{
      author={Lascoux, A.},
      author={Sch\"utzenberger, M.-P.},
       title={{Schubert polynomials and the {Littlewood}-{Richardson} rule}},
        date={1985},
     journal={Lett. Math. Phys.},
      volume={10},
      number={2--3},
       pages={111\ndash 124},
}

\bib{MacdonaldSchubertBook}{book}{
      author={Macdonald, I.G.},
       title={{Notes on Schubert Polynomials}},
   publisher={Universit\'e du Qu\'ebec \`a Montr\'eal},
        date={1991},
}

\bib{ManivelBook}{book}{
      author={Manivel, L.},
       title={{Symmetric functions, {Schubert} polynomials and degeneracy
  loci}},
      series={SMF/AMS Texts and Monographs},
   publisher={American Mathematical Society, Providence, RI; Soci\'et\'e{}
  Math\'ematique de France, Paris},
        date={2001},
      volume={6},
        ISBN={0-8218-2154-7},
        note={Translated from the 1998 French original by John R. Swallow,
  Cours Sp\'ecialis\'es, 3. [Specialized Courses]},
      review={\MR{1852463}},
}

\bib{MoralesPakPanova2019}{article}{
      author={Morales, A.H.},
      author={Pak, I.},
      author={Panova, G.},
       title={{Asymptotics of principal evaluations of Schubert polynomials for
  layered permutations}},
        date={2019},
     journal={Proc. Amer. Math. Soc.},
      volume={147},
       pages={1377\ndash 1389},
        note={arXiv:1805.04341 [math.CO]},
}

\bib{GrothendieckShenanigans2024}{article}{
      author={Morales, A.H.},
      author={Panova, G.},
      author={Petrov, L.},
      author={Yeliussizov, D.},
       title={{Grothendieck Shenanigans: Permutons from pipe dreams via
  integrable probability}},
        date={2025},
     journal={Adv. Math.},
      volume={480},
       pages={110510},
        note={arXiv:2407.21653 [math.PR]},
}

\bib{monical2022reduced}{article}{
      author={Monical, C.},
      author={Pankow, B.},
      author={Yong, A.},
       title={{Reduced word enumeration, complexity, and randomization}},
        date={2022},
     journal={Electron. J. Combin.},
      volume={29},
      number={2},
       pages={P2.46},
}

\bib{merzon2016determinantal}{article}{
      author={Merzon, G.},
      author={Smirnov, E.},
       title={{Determinantal identities for flagged Schur and Schubert
  polynomials}},
        date={2016},
     journal={Eur. J. Math.},
      volume={2},
       pages={227\ndash 245},
        note={arXiv:1410.6857 [math.CO]},
}

\bib{OEIS}{misc}{
      author={{OEIS Foundation Inc.}},
       title={The {O}n-{L}ine {E}ncyclopedia of {I}nteger {S}equences},
        date={2025},
         url={https://oeis.org},
        note={Published electronically at \url{https://oeis.org}},
}

\bib{Petrov2012GFF}{article}{
      author={Petrov, L.},
       title={{Asymptotics of Uniformly Random Lozenge Tilings of Polygons.
  Gaussian Free Field}},
        date={2015},
     journal={{Ann. Probab.}},
      volume={43},
      number={1},
       pages={1\ndash 43},
      eprint={1206.5123},
        note={arXiv:1206.5123 [math.PR].},
}

\bib{propp2003generalized}{article}{
      author={Propp, J.},
       title={Generalized domino-shuffling},
        date={2003},
     journal={Theoretical Computer Science},
      volume={303},
      number={2-3},
       pages={267\ndash 301},
        note={arXiv:math/0111034 [math.CO]},
}

\bib{ProppWilsonCP}{incollection}{
      author={Propp, J.},
      author={Wilson, D.},
       title={{Coupling from the past: a user's guide}},
        date={1998},
   booktitle={Microsurveys in discrete probability ({P}rinceton, {NJ}, 1997)},
      series={DIMACS Ser. Discrete Math. Theoret. Comput. Sci.},
      volume={41},
   publisher={AMS},
       pages={181\ndash 192},
}

\bib{romera-paredes2024FunSearch}{article}{
      author={Romera-Paredes, Bernardino},
      author={Barekatain, Mohammadamin},
      author={Novikov, Alexander},
      author={Balog, Matej},
      author={Kumar, M.~Pawan},
      author={Dupont, Emilien},
      author={Ruiz, Francisco J.~R.},
      author={Ellenberg, Jordan~S.},
      author={Wang, Pengming},
      author={Fawzi, Omar},
      author={Kohli, Pushmeet},
      author={Fawzi, Alhussein},
       title={Mathematical discoveries from program search with large language
  models},
        date={2024},
     journal={Nature},
      volume={625},
      number={7995},
       pages={468\ndash 475},
}

\bib{Sheffield2007GFF}{article}{
      author={Sheffield, S.},
       title={Gaussian free fields for mathematicians},
        date={2007},
     journal={Probab. Theory Relat. Fields},
      volume={139},
      number={3-4},
       pages={521\ndash 541},
        note={arXiv:math/0312099 [math.PR]},
}

\bib{stanley2017some}{article}{
      author={Stanley, R.},
       title={{Some Schubert shenanigans}},
        date={2017},
     journal={arXiv preprint},
        note={arXiv:1704.00851 [math.CO]},
}

\bib{Weigandt2020_bumpless}{article}{
      author={Weigandt, A.},
       title={{Bumpless pipe dreams and alternating sign matrices}},
        date={2021},
     journal={J. Comb. Theory Ser. A},
      volume={182},
       pages={105470},
        note={arXiv:2003.07342 [math.CO]},
}

\bib{ZinnJustin20096Vertex}{article}{
      author={Zinn-Justin, P.},
       title={{Six-Vertex, Loop and Tiling models: Integrability and
  Combinatorics}},
        date={2009},
     journal={arXiv preprint},
        note={arXiv:0901.0665 [math-ph]},
}

\end{biblist}
\end{bibdiv}
\bibliographystyle{alpha}

\end{document}